\documentclass[a4paper,10pt]{article}
 \usepackage{a4wide}

\usepackage[utf8]{inputenc}
\usepackage{amsfonts,amsmath,amssymb, mathtools, amsthm}
\usepackage{enumerate}
\usepackage[usenames,dvipsnames]{xcolor}
\mathtoolsset{showonlyrefs}
\usepackage{cite}
\usepackage{authblk}
\usepackage{hyperref}

\newtheorem{theorem}{Theorem}
\newtheorem{assumption}{Assumption}
\newtheorem{lemma}[theorem]{Lemma}
\newtheorem{proposition}[theorem]{Proposition}
\newtheorem{definition}[theorem]{Definition}
\newtheorem{corollary}[theorem]{Corollary}
\newtheorem{remark}[theorem]{Remark}

\newcommand{\R}{\mathbb{R}}
\newcommand{\N}{\mathbb{N}}
\newcommand{\II}{\mathcal{I}}
\newcommand{\KK}{\mathcal{K}}

\newcommand{\UU}{\mathcal{U}}

\renewcommand{\O}{\mathrm{O}}
\newcommand{\RR}{\mathsf{R}}
\newcommand{\RRR}{\mathcal{R}}
\newcommand{\XX}{\mathcal{X}}

\newcommand{\T}{\bar{T}}

\DeclareMathOperator{\e}{e}

\newcommand{\diag}{\mathsf{D}}
\newcommand{\Kp}{\mathsf{K}}
\newcommand{\uKp}{\underline{\Kp}}

\newcommand{\dt}{\Delta}
\newcommand{\xhat}{\hat{x}}

\newcommand{\dd}{\mathrm{d}}

\renewcommand{\epsilon}{\varepsilon}
\DeclareMathOperator{\sat}{sat}
\DeclareMathOperator{\rank}{rk}
\DeclareMathOperator{\inv}{inv}

\newcommand{\startmodif}{\begingroup}
\newcommand{\stopmodif}{\endgroup}

\newcommand{\startmodifVA}{\begingroup}

\newcommand{\HH}{\mathcal{H}}
\newcommand{\lip}{\kappa}

\renewcommand{\S}{\mathbb{S}}

\newcommand{\defeq}{\vcentcolon=}

\renewcommand{\le}{\leqslant}
\renewcommand{\ge}{\geqslant}
\renewcommand{\leq}{\leqslant} 
\renewcommand{\geq}{\geqslant}

\newcommand{\lang}{\left\langle}
\newcommand{\rang}{\right\rangle}

\DeclareMathOperator{\codim}{codim}

\renewcommand{\emptyset}{\varnothing}

\DeclareMathOperator{\Id}{Id}
\newcommand{\invert}{\mathrm{inv}}

\renewcommand{\mid}{\;:\;}

\title{Exponential stabilizability and observability at the target
imply semiglobal exponential stabilizability
by templated
output
feedback
}
\author[1]{Vincent Andrieu}
\author[2]{Lucas Brivadis}
\author[3]{Jean-Paul Gauthier}
\author[4]{Ludovic Sacchelli}
\author[1]{Ulysse Serres}

\affil[1]{\small Univ. Lyon, Universit\'e Claude Bernard Lyon 1, CNRS, LAGEPP UMR 5007, 43 bd du 11 novembre 1918, F-69100 Villeurbanne, France (emails: \texttt{vincent.andrieu@univ-lyon1.fr, ulysse.serres@univ-lyon1.fr})}

\affil[2]{\small Université Paris-Saclay, CNRS, CentraleSupélec, Laboratoire des Signaux et Systèmes, 91190, Gif-sur-Yvette, France
(email: \texttt{lucas.brivadis@centralesupelec.fr})
}%

\affil[3]{\small  Universit\'e de Toulon, Aix Marseille Univ, CNRS, LIS, France (email: \texttt{jean-paul.gauthier@univ-tln.fr}}

\affil[4]{\small Inria, Université Côte d’Azur, CNRS, LJAD, MCTAO team, Sophia Antipolis, France
 (email: \texttt{ludovic.sacchelli@inria.fr})}

\date{\today}

\begin{document}

\maketitle

\begin{abstract}
For nonlinear analytic control systems, we introduce a new paradigm for dynamic output feedback stabilization. We propose to periodically sample the usual observer based control law,
and to reshape it so that it coincides with a ``control template'' on each time period. By choosing a control template making the system observable, we prove that this method allows to bypass the uniform observability assumption that is used in most nonlinear separation principles.
We prove the genericity of control templates by adapting a universality theorem of Sussmann.
\end{abstract}

\section{Introduction}

The concept of dynamic output feedback holds a central position in control theory. When confronted with a model of a physical phenomenon encapsulating a controlled dynamical system and observable parameters, it is crucial to design a control strategy that ensures stability at a particular operating condition. As a result, the problem of output feedback stabilization has gathered considerable attention from researchers.
In the 1990s and 2000s, many researchers tackled this issue, continuously expanding the class of systems for which it was possible to design an output feedback loop.
In \cite{AndrieuPraly2009}, an overview on existing techniques is given.

The most commonly accepted idea for synthesizing a control law is to split this problem into two subtasks: the first being to stabilize the system using state feedback, and the second being to reconstruct the system's state from the available measurements. 
This approach then involves combining the state reconstruction procedure (a dynamic asymptotic observer) with the obtained state feedback to form a dynamic output feedback scheme.
\startmodif
This is the approach of the seminal articles by Teel and Praly \cite{TeelPraly1994, TeelPraly1995}
as well as Jouan and Gauthier \cite{jouan, jouan1996finite}, who obtained so-called nonlinear separation principles.
With few exceptions that we discuss later, most such existing results
\stopmodif
require that the observability property allowing state estimation is uniform over all time-varying inputs:
\begin{description}
    \item[o1)] For any input, two different states lead to different measured outputs.
\end{description}
Unfortunately, it was demonstrated by Gauthier and Kupka in \cite{Gauthier_book} that this observability requirement is generically not satisfied if the output dimension is less than or equal to the input dimension.
\startmodif
Therefore, other feedback methods have been developed in the literature to overcome the uniform observability condition.
In most cases, the solution comes from allowing non-stationary (i.e. time-dependent) feedbacks.
This approach was originally proposed by Coron in \cite{coron1994stabilization} a year after \cite{TeelPraly1994}, by achieving local asymptotic stabilization through output feedback for general nonlinear systems employing a (periodic) time-varying strategy. Assuming the target is the origin and is an equilibrium of the system with null input, Coron's results necessitate the following two weak notions of observability:
\begin{description}
    \item[o2.i)] Given two different states, an input exists which leads to different outputs.
    \item[o2.ii)] If the input and output are identically null, then the state is at the target.
\end{description}
\stopmodif
The results obtained in \cite{coron1994stabilization} are local.
In other words, the basin of attraction for the output feedback loop only contains a neighborhood of the origin, and its size cannot be 
\startmodif
prefixed.
Inspired by this work, Shim and Teel obtained a semiglobal result by assuming (essentially):
\begin{description}
    \item[o3)] There exists an input such that any two different states lead to different outputs.
\end{description}
Their strategy has two modes that are periodically activated: in a first mode, a control making the system observable in a sufficiently short time, so that a high-gain observer can be designed to quickly estimate the state; in a second mode, the stabilizing control law based on the observer is applied.
Since the input making the system observable is fixed and periodically applied to the system, asymptotic convergence may be prevented, but still allowing a practical result (i.e. convergence in an arbitrarily small neighborhood of the target).

More recently, 
observation and control of non-uniformly observable systems is gathering attention because of appearances of the issue in numerous modern applications \cite{surroop2020adding, Rapaport, ludo, 9039752}.
The two modes strategy introduced in \cite{shim2003asymptotic} has encouraged researchers to develop dynamic output feedbacks in the framework of hybrid systems \cite{goedel2012hybrid}. This method has proven to be useful for systems presenting specific structures \cite{10264115, 9683664, BRIVADIS20239517}.
The present paper falls directly in line with the seminal efforts of \cite{coron1994stabilization} and \cite{shim2003asymptotic}, but we adopt a different viewpoint. Although our output feedback is periodic and realized through hybrid dynamics, we do not separate the control in two modes: each time step simultaneously allows to observe and control the system.
To do so, we propose a ``templated'' output feedback strategy, which aims to generalize ''sampled'' output feedback.
On each time interval, we sample the value of the observer-based control law at the beginning of the interval and, instead of ``holding'' it constant, we modulate it into the shape of a control template.
We pick as our control template an input making the system observable, and such that any rescaling or isometry of the template maintains this observability property.

To the best of our knowledge, this control strategy is new, and helps in the context of non-uniformly observable systems.
On the one hand, a high-gain observer can then be employed to estimate the state sufficiently fast on each time step.
On the other hand, since we require that any rescaling of the template still induce observability, it leads us to assume observability of the null input, i.e.
\begin{description}
        \item[o4)] If the input is identically null,  any two different states lead to different outputs.
\end{description}
Obviously, we have $o1)\Rightarrow o4) \Rightarrow o3) \Rightarrow o2.i)$ and $o4)  \Rightarrow o2)$. Note that assumption $o4)$ naturally excludes the case of systems that are unobservable at the target point (see e.g. \cite{brivadis2021new}).
Nevertheless, our observability condition is generic (i.e. satisfied for almost all nonlinear systems).
It may appear that the existence of a control template, on which the whole strategy hinges, is a restrictive property, but it is in fact equivalent to $o4)$, as far as analytic systems are concerned. This is to be put in perspective with the existence of inputs that make such systems observable, which was proved to be always the case. The proof we employ actually follows from an argument of this type made by Sussmann in 1979 \cite{MR0541865}.

\stopmodif

\subsubsection*{Notations}

Let $n$, $m$ and $p$ be positive integers, $\II$ be an interval of $\R$, and $\XX$ be an open set of $\R^n$ endowed with the Euclidean norm $|\cdot|$.
We denote by $PC^0(\II, \R^p)$ the space of piecewise continuous functions from $\II$ to $\R^p$, by $C^k(\XX, \R^m)$ the space of $k$-times continuously differentiable functions
from $\XX$ to $\R^m$ for $k\in\N\cup\{+\infty\}$, and by $C^\omega(\XX, \R^m)$ the space of analytic functions
from $\XX$ to $\R^m$.
For any $u\in PC^0(\II, \R^p)$, we denote by $\|u\|_{\infty} \defeq \sup_{t\in\II}|u(t)|$ the uniform norm of $u$.
If $\II$ is not open, we shall write that $u\in C^k(\II, \R^p)$ for some $k$ if $u$ is the restriction to $\II$ of some $C^k$ function defined on a larger open interval.
We denote by 
$\left.\frac{\dd^k}{\dd t^k}\right|_{t=t_0} u(t)$ the derivative of order $k$ of $u$ evaluated at $t_0\in\II$.
For any $h\in C^1(\R^n, \R^m)$, we denote by $\frac{\partial h}{\partial x}(x)$ the differential of $h$ evaluated at $x\in\R^n$.
If $h\in C^k(\R^n, \R^m)$, we denote by $h^{(k)}(x)$ the differential of order $k$ of $h$ evaluated at $x$.
For any set $\XX$ in a topological space, let
$\mathring\XX$,  $\overline\XX$, and $\partial\XX$ denote its interior, its closure, and its boundary, respectively.
For any $x\in\R^n$ and $r>0$, we denote by $B_{\R^n}(x, r)$ the open ball of $\R^n$ centered at $x$ of radius $r$.
We denote by $[x]_{m:(m+p)}\in\R^p$ the vector composed of coordinates of $x$ from index $m$ to $m+p$.
We denote by $\O(p)$ the orthogonal group of $\R^p$, endowed with the induced norm $\|\cdot\|$.

\section{Templated output feedback}
\label{sec:template}

\subsection{Problem statement and main assumptions}
\label{sec:def-Hq}

Consider an analytic control system
\begin{equation}\label{syst}
\left\{
\begin{aligned}
    &\dot x = f(x, u)\\
    &y = h(x)
\end{aligned}
\right.
\end{equation}
where $x$ in $\R^n$ is the state of the system,
$u$ in $\R^p$ is the control input, $y$ in $\R^m$ is the measured output
and $f:\R^n\times\R^p\to\R^n$ and $h:\R^n\to\R^m$ are analytic maps.
Since $f$ is continuous and locally Lipschitz continuous with respect to its first variable, uniformly with respect to the second, according to the Cauchy-Lipschitz theorem,
the Cauchy problem $\dot{x}(t)=f(x(t),u(t))$, $x(0)=x_0$
admits a unique maximal solution that we denote by $t\mapsto X(x_0,t,u)$ 
for any input $u\in PC^0(\R_+,\R^p)$.
Moreover, we denote by $T_e(x_0, u)\in\R_+^*\cup\{+\infty\}$ its maximal time of existence.
If $u\in PC^0(\II,\R^p)$ for some interval $\II\subset\R_+$ containing $0$, we abuse the notation $X(\cdot, t, u)$ for $t\in\II$ to denote any solution $X(\cdot, t, v)$ such that $v_{|\II}=u$.

In order to introduce notions of observability of \eqref{syst}, let us define the following functions by induction.
Let $H_0=h$, and for any positive integer $k$,
let $H_k:\R^n\times (\R^p)^{k}\to \R^m$
be such that for all $(x, \sigma)\in \R^n\times (\R^p)^{k+1}$,
\[
\begin{aligned}
H_{k+1}(x,\sigma_0,\sigma_1,\dots,\sigma_{k})=&
\frac{\partial H_k}{\partial x}(x,\sigma_0,\sigma_1,\dots,\sigma_{k-1})f(x,\sigma_0)
\\
&+
\sum_{i=0}^{k-1}
\frac{\partial H_k}{\partial \sigma_i}(x,\sigma_0,\sigma_1,\dots,\sigma_{k-1})\sigma_{i+1}
\end{aligned}
\]
For any positive time $T\geq 0$ and any integer $k\in\N$, we set $\mathcal{H}_k:\R^n\times[0, T]\times  C^{k}([0, T], \R^p)\to \R^{m(k+1)}$ to be the collection of $H_j$, $j$ from 0 to $k$, evaluated on the jets of an input $u$, that is
$$
\left[\mathcal{H}_k(x,t,u)\right]_{jm+1:(j+1)m}=H_j(x,u(t),u^{(1)}(t),\dots, u^{(j-1)}(t)),\qquad 0\leq j\leq k.
$$
If we consider an input\footnote{
Note that one can define $\HH_0$ and $\HH_1$ even for discontinuous inputs, and that for $k>1$, one only needs $u$ to be differentiable of order $k-1$ to define $\HH_k$. However, we assume $u$ to be $C^k$ to ease the reading.
In practice, in this paper, we only apply $\HH_k$ to inputs that are $C^\infty$ on $[0, T]$.
} $u\in C^k([0,T],\R^p)$, and, for any $k>0$, $\sigma:=(u(0),u'(0),\dots,u^{(k-1)}(0))\in (\R^{p})^k$, it is clear that
$H_k(x,\sigma)=
\left.
\frac{\dd^{k}}{\dd t^k}
\right|_{t=0}h(X(x_0,t,u)) $.
In other words,
$\HH_k(x, t, u)$ contains the jets at time $0$, of order $0$ up to $k$, of the output $y$ of the solution to \eqref{syst} initialized at $x$ and with input $u$. Note that since $f$ and $h$ are analytic, so are $x\mapsto H_k(x, \sigma)$ and $x\mapsto \HH_k(x, t, u)$ for all $k\in\N$, all $t\in[0, T]$, all $\sigma\in{\R^p}^k$ and all $u\in C^k([0, T], \R^p)$. Moreover, $(x,u)\mapsto\HH_k(x, t, u)$ is locally Lipschitz continuous for all $t\in[0, T]$.

\begin{definition}[Observability) (see, e.g., \cite{MR1862985}]
    We say that system \eqref{syst} is:
    \begin{enumerate}[(i)]
        \item \emph{observable} 
        over a set $\XX\subset\R^n$
        in time $\tau\in[0, T]$ for the input $u\in PC^0([0, T], \R^m)$ if, for all initial conditions $x_0\neq \tilde x_0\in\XX$, there exists $t\in[0, \bar \tau]$ such that
        $h(X(x_0, t, u))\neq h(X(\tilde x_0, t, u))$,
        where $\bar\tau$ is such that $X(x_0, s, u)$ and $X(\tilde x_0, s, u)$ are well-defined and lie in $\XX$ for all $s\leq\tau$.
        
        \item \emph{differentially observable} over a set $\XX\subset\R^n$ of order $k\in\N$
        for the input $u\in C^{k}([0, T], \R^m)$ if
        $x\mapsto \HH_k(x, 0, u)$ is injective over $\XX$, and \emph{strongly} differentially observable if moreover $x\mapsto \HH_k(x, 0, u)$ is an immersion over $\XX$, that is, $\frac{\partial \HH_k(\cdot, 0, u)}{\partial x}(x)$ is injective for all $x\in\XX$.
    \end{enumerate}
\end{definition}

Clearly, differential observability implies observability in any positive time.
Also, by definition of $\HH$, differential (resp. strong differential) observability of some order $k_0$ implies differential (resp. strong differential) observability of any order $k\geq k_0$.
Since $f$ and $h$ are analytic, we also have the following lemma showing that observability implies differential observability on compact sets for analytic inputs.

\begin{lemma}\label{lem:analytic}
If system \eqref{syst} is observable over a compact set $\KK_x\subset\R^n$
in some positive time $\tau$ for some input $u\in C^\omega([0, T], \R^m)$,
then
there exists $k\in\N$ such that
\eqref{syst} is also differentially observable over $\KK_x$ of order $k$.
\end{lemma}
The proof of Lemma \ref{lem:analytic} is postponed to Appendix \ref{sec:tec2}.
Let $\XX\subset\R^n$ be an open set that is fixed for the rest of the paper. Our main objective is to stabilize system \eqref{syst} by means of a dynamic output feedback semi-globally over $\XX$, that is, with a basin of attraction arbitrarily large within $\XX$. In order to do so, we make three main assumptions.
The first two are observability assumptions for the null input $u=0$,
while the third one is a state-feedback stabilizability assumption.

\begin{assumption}\label{ass:obs0}
The null input $u=0$ makes system \eqref{syst} observable over $\XX$ in some time.
\end{assumption}

\begin{assumption}\label{ass:immersion0}
    For all $x\in \XX$, there exists $k_x\in\N$ such that $\frac{\partial \HH_{k_x}(\cdot, 0, 0)}{\partial x}(x)$ is injective.
\end{assumption}

Note that if system \eqref{syst} is strongly differentially observable over $\XX$ of some order for the null input $u=0$, then Assumptions \ref{ass:obs0} and \ref{ass:immersion0} clearly holds. The following converse result also holds.

\begin{lemma}\label{lem:hyp12}
Let $\KK_x$ be a compact subset of $\XX$.
Assumptions \ref{ass:obs0} and \ref{ass:immersion0} together imply that
system \eqref{syst} is strongly differentially observable over $\KK_x$ of some order for the null input.

\end{lemma}
The proof of Lemma \ref{lem:hyp12} is postponed to Appendix \ref{sec:tec2}.
\begin{assumption}\label{ass:GES}
    There exists a feedback law $\lambda:\R^n\to\R^p$
    that is locally Lipschitz continuous on $\XX$
    such that \eqref{syst}
    in closed-loop with $u = \lambda(x)$
    is
    locally exponentially stable (LES) at the origin,
    with a basin of attraction containing $\XX$.
\end{assumption}

With no loss of generality, we assume that $\lambda(0) = 0$ and $h(0)=0$ in the following.
Note that in the case where $\XX=\R^n$, the vector field $x\mapsto f(x, \lambda(x))$ is globally asymptotically stable (GAS) and LES at the origin, but not necessarily globally exponentially stable (GES).
While the historic results on output feedback stabilisation only require asymptotic stability, exponential stability of the origin is necessary to achieve feedback stabilisation  via sampling techniques (even state feedback), as illustrated by Section~\ref{S:stabilisation_templated}.

According to the converse Lyapunov function theorem \cite[Remark 2.350]{PralyBresch2022a},
there exist $V\in C^1(\R^n, \R_+)$ 
such that for all compact set $\KK_x\subset\XX$, there exists
four positive constants $(\alpha_i)_{1\leq i\leq 4}$ such that, for all $\xi\in\KK_x$,
\begin{equation}\label{eq:Vexp}
\begin{aligned}
    \alpha_1 |\xi|^2 \leq V(\xi) &\leq \alpha_2 |\xi|^2,\\
    \left|\frac{\partial V}{\partial x}(\xi)\right| &\leq \alpha_3 |\xi|,\\
    \frac{\partial V}{\partial x}(\xi)f(\xi, \lambda(\xi)) &\leq - \alpha_4 |\xi|^2.
\end{aligned}
\end{equation}

\medskip

In order to stabilize \eqref{syst} by means of a dynamic output feedback, we employ an observer-based strategy, that is, we combine an observer (an online estimation algorithm of the state based on the measurement of the output) with the state-feedback law $\lambda$ \cite{AndrieuPraly2009}.
This strategy has been extensively studied under uniform observability assumptions, that is, observability for all inputs \cite{TeelPraly1994, TeelPraly1995, jouan, jouan1996finite}. Essentially, uniformly observable systems that are stabilizable by means of a state feedback are also semi-globally stabilizable by means of a dynamic output feedback.
However, uniform observability assumptions are not generic \cite[Chapter 3]{Gauthier_book}, contrary to the weaker Assumptions \ref{ass:obs0} and \ref{ass:immersion0}, since strong differential observability of order $2n+1$ for $u=0$ is generic by \cite[Chapter 4, Theorem 2.2]{Gauthier_book}.
Our goal is to obtain a nonlinear separation principle in line with \cite{TeelPraly1994, TeelPraly1995, jouan, jouan1996finite} in the case where the uniform observability assumption is replaced by Assumptions \ref{ass:obs0} and \ref{ass:immersion0}, i.e., observability assumptions only for the null input. Doing so, we obtain a generic nonlinear separation principle.
The price to pay
is that: $(i)$ in Assumption \ref{ass:GES}, the stabilizing state-feedback law is supposed to make the system LES instead of locally asymptotically stable (LAS) as in the uniformly observable case \cite{TeelPraly1994, TeelPraly1995, jouan, jouan1996finite}; $(ii)$ we allow the dynamic output feedback to have hybrid dynamics.

\subsection{Hybrid dynamic output feedback design}
\label{S:Hybrid dynamic output feedback design}
Since we aim at semiglobal dynamic output feedback stabilization of \eqref{syst} in $\XX$, we design a hybrid dynamic output feedback law that stabilizes \eqref{syst} with a basin of attraction containing an arbitrarily large compact set $\KK_x\subset \XX$, and that prevents the state to escape some larger compact $\KK_x'\subset\XX$. With no loss of generality, we assume that
$\KK_x' = \{x\in\XX\mid V(x) \leq \bar r\}$ for some fixed $\bar r>0$, and that $\KK_x\subset\mathring\KK_x'$.
In the rest of the paper, $\KK_x$ and $\KK_x'$ are now fixed.
Our strategy relies on the use of a \emph{control template}.
Define $\bar\lambda = \max_{x\in \KK_x'}|\lambda(x)|$.
\begin{definition}[Control template]
    An input $v^*\in C^\infty([0, T], \R^p)$ is said to be a \emph{control template}
    of order $q\in\N$  
    if
    $v^*(0) = (1,0,\dots,0)$ and
    the map $x\mapsto \mathcal{H}_{q}(x,t,\mu \RR v^*)$ is an injective immersion over $\KK_x'$, for all $t\in [0,T]$, all $\mu \in [0,\bar\lambda]$ and all $\RR\in \O(p)$.
\end{definition}

\startmodif
Roughly speaking,
a control template is an input making the system strongly differentially observable over $\KK_x'$, and such that any rescaling ($\mu$) or isometry ($\RR$) of this input preserves this observability property.
\stopmodif
The key property of control templates is given by the following lemma.

\begin{lemma}\label{Ass_HG}
Let $v^*\in C^\infty([0, T], \R^p)$ be a control template of order $q\in\N$.
There exists $\phi\in C^0(\R^{m(q+1)}\times\R^p\times\R, \R^n)$ 
such that for all $x_0\in \KK_x$ and all  $(t,\mu, \RR)\in[0,T]\times[0, \bar \lambda]\times\O(p)$ the solution $(x, y)$ to system \eqref{syst} initialized at $x_0$ with input $u=\mu \RR v^*$ satisfies the following.
For all $t\in[0, T]$ such that $x(s)\in\KK'$ for all $s\in[0, t]$,
\begin{align}\label{eq:phi}
x(t)&=\phi(y(t),\dots,y^{(q)}(t),t,\mu,\RR).
\end{align}
Moreover, there exists a positive constant $\lip_\phi$ such that for 
all $(z_a,z_b)\in(\R^{m(q+1)})^2$ and all $(t,\mu,\RR)\in[0,T]\times[0, \bar \lambda]\times \O(p)$,
\begin{align}\label{eq:lip_phi}
|\phi(z_a,t,\mu,\RR)-\phi(z_b,t,\mu,\RR)|&\leq \lip_\phi|z_a-z_b|,
\end{align}
\end{lemma}
The proof of Lemma \ref{Ass_HG} is postponed to Appendix \ref{sec:tec2}.
In Theorem \ref{lem:sus_bis}, we state that control templates are generic under Assumptions \ref{ass:obs0} and \ref{ass:immersion0}. For the time being, let us focus on the design of a hybrid dynamic output feedback, based on the existence of a control template $v^*\in C^\infty([0, T], \R^p)$. Let $\phi$ be as in Lemma \ref{Ass_HG}.
For all $u_0\in\R^p$, define the nonempty set 
$$
\RRR_0(u_0) = \Big\{\RR\in\O(p)\mid \RR\begin{pmatrix}
|u_0|\\0\\\vdots\\0
\end{pmatrix} = u_0\Big\}.
$$
Let $\sat\in C^\infty(\R^n, \R^n)$ 
be such that $\sat(x) = x$ for all $x\in\KK_x'$ and $|\sat(x)| \leq 2\max_{x'\in\KK'_x}|x'|$ for all $x\in\R^n$.
Let $(c_i)_{0\leq i\leq q}\in\R^{q+1}$, $\dt>0$ and $\theta>0$ be observer parameters to be tuned later on.
We propose the following hybrid dynamic output feedback:
\begin{equation}\label{eq:closed-loop}
\left\{\begin{aligned}
&\left.\begin{aligned}
&\dot x = f(x, \mu \RR v^*(s))
\\
&\dot z_0
= z_1 + \theta c_0(h(x)-z_0)\\
&\dot z_1 = z_2 + \theta^2 c_1(h(x)-z_0) \\
&\ \ \vdots\\
&\dot z_{q-1} = z_{q} + \theta^{q} c_{q-1}(h(x)-z_0) \\
&\dot z_{q} = H_{q+1}(\sat(\phi(z, s, \mu, \RR)),\mu\RR v^*(s),\dots,\mu\RR(v^*)^{(q)}(s))\quad
\\&\qquad+ \theta^{q+1} c_{q}(h(x)-z_0)
\\
&\dot s = 1
,\quad\dot \mu = 0
,\quad\dot \RR = 0
\end{aligned}\right\}
&s\in[0, \Delta],
\\
\\
&\left.\begin{aligned}
&x^+ = x
\\
&z^+ = \HH_q(\sat(\phi(z, \Delta, \mu, \RR)), 0, \mu^+\RR^+v^*)\\
&s^+ = 0
\\
&\mu^+ = |\lambda(\sat(\phi(z, \Delta, \mu, \RR)))|
\\
&\RR^+ \in \RRR_0(\lambda(\sat(\phi(z, \Delta, \mu, \RR))))
\end{aligned}\hspace{2.8cm}\right\}
&s = \Delta.
\end{aligned}\right.
\end{equation}

\begin{figure}
    \centering
    \includegraphics[trim=100 140 150 120,clip,width=0.8\linewidth]{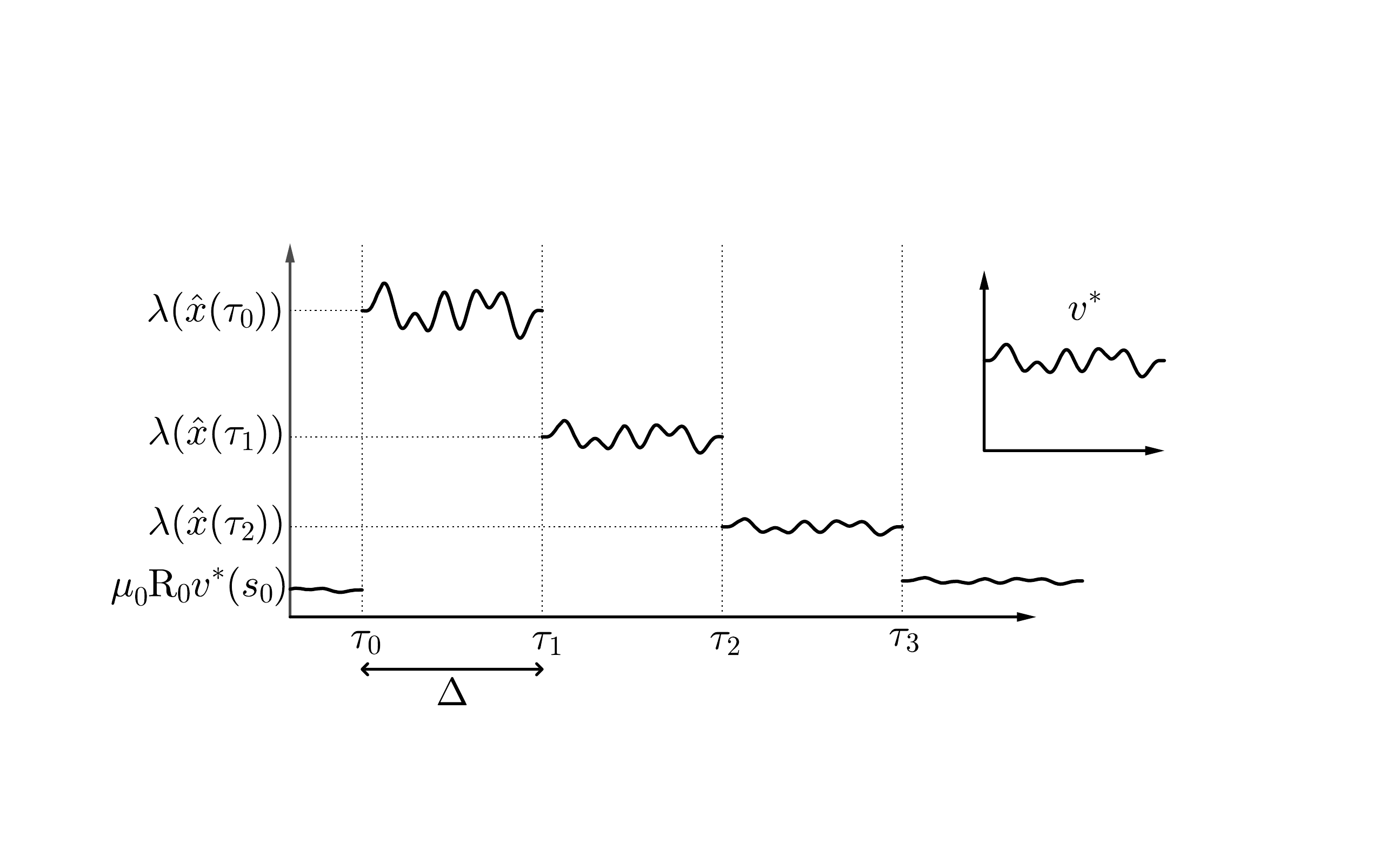}
    \caption{
    \startmodif
    Trajectory of the input $u=\mu\RR v^*$ applied to the system when using the templated output feedback strategy of \eqref{eq:closed-loop}.
    To lighten the notations, here we write $\hat x:= \phi(z, \dt, \mu, \RR)$.
    The illustration corresponds to the case of a one-dimensional input.
    After each jump, $u(\tau_i)=\lambda(\xhat(\tau_i))$. Then, over each time-interval, the input follows the shape of the control template $v^*$. Note that the input is also rescaled over each interval by $|\lambda(\xhat(\tau_i))|$, in order to guarantee that it remains close to $\lambda(\xhat)$ and $u\to0$ as $\xhat\to0$.
    \stopmodif
    }
    \label{fig}
\end{figure}

\startmodif

Let us explain the above output feedback dynamics.
The jump times are periodically triggered every time the timer $s$ reached $\dt$. Over each interval of length $\dt$, the control law applied to the system is $u:=\mu \RR v^*$, making the system observable due to the definition of control templates.
We employ a high-gain observer during the flow ($z$ variable) to estimate the state thanks to the knowledge of the output $h(x)$ and input $u$. The state estimation is given by $\xhat := \phi(z, \Delta, \mu, \RR)$, since $\phi$ satisfies \eqref{eq:phi} and $(z_0,\dots,z_q)$ approaches $(y,\dots,y^{(q)})$.
At each jump,
the scaling parameter $\mu$ and the isometry $\RR$ are updated, such that at the beginning of each time period, $\mu\RR v^*(0) = \lambda(\xhat)$. Since $\dt$ is to be chosen small enough, this guarantees that $u$ remains close to $\lambda(\xhat)$. Note that the $z$ dynamics also jump, because the immersion $\HH_q$ depends on the input $u$, which jumps from $\mu\RR v^*(\dt)$ to $\mu^+\RR^+v^*(0)$.
Figure~\ref{fig} illustrates the trajectory of the control law $\mu\RR v^*$ during the stabilization procedure.
Remark that the amplitude of the input over each time interval is proportional to its value at the beginning of the interval: this is the role of the scaling parameter $\mu$, and it is crucial to guarantee that $u\to0$ when $\xhat\to0$.

Note that in the case where $v^*$ is constant, the above closed-loop simply consists in a sample-and-hold of the dynamic output feedback based on the usual high-gain observer.
Indeed, since $v^*(0) = (1,0,\dots,0)$ and due to the definition of $\RRR_0$, 
the input $u:=\mu\RR v^*$ is piecewise constant and
the dynamics read as follows (with an abuse of notations on $\phi$ to make it depend on $u$ instead of $(\mu,\RR)$):

\begin{equation}\label{eq:closed-loop_sample}
\left\{\begin{aligned}
&\left.\begin{aligned}
&\dot x = f(x, v)
\\
&\dot z_0
= z_1 + \theta c_0(h(x)-z_0)\\
&\dot z_1 = z_2 + \theta^2 c_1(h(x)-z_0) \\
&\ \ \vdots\\
&\dot z_{q-1} = z_{q} + \theta^{q} c_{q-1}(h(x)-z_0) \\
&\dot z_{q} = H_{q+1}(\sat(\phi(z, s, v)),v)\quad
\\&\qquad+ \theta^{q+1} c_{q}(h(x)-z_0)
\\
&\dot s = 1
,\quad\dot v = 0
\end{aligned}\hspace{1cm}\right\}
&s\in[0, \Delta],
\\
\\
&\left.\begin{aligned}
&x^+ = x
\\
&z^+ = \HH_q(\sat(\phi(z, \Delta, v)), 0, v^+)\\
&s^+ = 0
\\
&v^+ = \lambda(\sat(\phi(z, \Delta, v)))
\end{aligned}\hspace{0.95cm}\right\}
&s = \Delta.
\end{aligned}\right.
\end{equation}

\stopmodif

\subsection{Main results}\label{sec:results}

\startmodif
Given the assumption of the existence of a control template, our first result is that 
\startmodifVA
the output feedback strategy proposed above successfully achieves semiglobal dynamic output feedback stabilization under the assumptions of exponential stabilizability and 
\startmodif
strong differential
\stopmodif
observability at the target.
\stopmodif
\stopmodif

\begin{theorem}[Output feedback stabilization theorem]\label{th:main}
    Suppose that Assumptions \ref{ass:obs0}, \ref{ass:immersion0} and \ref{ass:GES} are satisfied. 
    \startmodifVA
    Assume there exists $v^*\in C^\infty([0, T], \R^p)$ a control template of order $q\in\N$.
    \stopmodif
    Then,
    there exist 
    \startmodif
    $\dt^*\in(0, T]$ 
    \stopmodif
    such that for all $\dt\in(0, \dt^*)$,
    for all $\KK_z\subset(\R^{m})^{q+1}$,
    for all $(c_i)_{0\leq i\leq q}\in\R^{q+1}$ that are coefficient of a Hurwitz polynomial,
    there exists $\theta^*\geq1$ such that for all $\theta>\theta^*$,
    the set $\{0\}\times\{0\}\times[0, \dt]\times\{0\}\times\O(p)$
    is LES for system \eqref{eq:closed-loop}, with basin of attraction containing
    $\KK_x\times\KK_z\times[0, \dt]\times[0, \bar\lambda]\times\O(p)$.
\end{theorem}

The proof of Theorem \ref{th:main} is the goal of Section \ref{sec:thmain}.
In the case where the control template is constant, this result shows that sampling-and-holding the control periodically over small time intervals is sufficient to achieve a nonlinear separation principle. We exploit the same high-gain observer based strategy than the literature \cite{TeelPraly1994, TeelPraly1995, jouan, jouan1996finite} that used to require uniform observability, whereas we only rely on observability at the target.

Because control templates are not necessarily constant (for example, almost all bilinear systems admit constant inputs making unobservable, as roots of the characteristic polynomial of the observability matrix \cite[Theorem 2.6]{brivadis2023output}),
our method proposes a generalization of the sample-and-hold strategy (namely, \eqref{eq:closed-loop} extends \eqref{eq:closed-loop_sample}). After sampling the control (the input is computed at each sampling time by composing the stabilizing feedback law with the observer), we propose, instead of ``holding'' it constant during a time $\dt$, to follow the shape of the control template $v^*$ from this starting point (up to a dilation $\mu$ and an isometry $\RR$ that preserve observability).

The natural remaining question is the existence of control templates, and how generic they are.
We propose to show their genericity with the following second main result, that is an extension of Sussmann's universality theorem \cite{MR0541865}.

\begin{theorem}[Universality theorem]\label{lem:sus_bis}
Suppose that Assumptions \ref{ass:obs0} and \ref{ass:immersion0} are satisfied.
Let $\KK$ be a compact subset of $\XX$.
Let $\T>0$ and let $\UU_{\KK}$ be the set of inputs $v$ in $C^\infty([0, \T], \R^p)$ (endowed with the compact-open topology)
for which there exist $T\in (0,\T]$, $q$ integer, such that the map $x\mapsto \mathcal{H}_{q}(x,t,\mu \RR v)$ is an injective immersion over $\KK$, for all $t\in [0,T]$, all $\mu \in [0,1]$ and all $\RR\in \O(p)$. Then
\begin{enumerate}[(i)]
    \item $\UU_{\KK}$  contains a countable intersection of open and dense subsets of $C^\infty([0, \T], \R^p)$;
    \label{sus1}
    \item the restriction of $\UU_{\KK}$ to analytic inputs, $\UU_{\KK}\cap C^{\omega}([0, \T], \R^p)$, is dense in $C^\infty([0, \T], \R^p)$. 
\end{enumerate}
\end{theorem}
The proof of Theorem \ref{lem:sus_bis} is the goal of Section \ref{sec:thuniv}.
It is an extension of the universality theorem of Sussmann \cite{MR0541865}.
In short, we use the stronger assumption that the null input makes the system strongly differentially observable on $\KK$ (while \cite{MR0541865} relies on differential observability only),
under which we prove:
$(i)$ the immersion property in addition to the injectivity;
$(ii)$ uniformity of the genericity with respect to parameters lying in a compact set.
The choice of compact-open topology for this result is completely natural and follows from \cite{MR0541865}. Notions on this topology will be recalled when necessary. See also \cite[Chapter 3]{engelking1989general} for a general reference on the topic.

With Theorem \ref{lem:sus_bis}, control templates can be obtained in the following manner.
Let $\KK = \KK_x'$ and $v_\mathrm{ref}$ be the constant input equal to $(2\bar\lambda,0,\dots,0)$. By genericity of $\UU_{\KK}$, arbitrarily choose $v\in \UU_{\KK}$ such that $\|v-v_\mathrm{ref}\|_\infty < \bar\lambda$. Then $|v(0)|>\bar\lambda$.
Define $\mu_\mathrm{ref} = \frac{1}{|v(0)|}$
and pick $\RR_\mathrm{ref}\in\O(p)$ such that $\RR^{-1}_\mathrm{ref} \in \RRR_0(v(0))$.
Set $v^* = \mu_\mathrm{ref}\RR_\mathrm{ref}v|_{[0, T']}$.
Then $v^*(0) = (1,0,\dots,0)$, and by definition of $\UU_{\KK}$, $v^*$ is a control template. Note that, as any genericity result based on transversality theory, Theorem \ref{lem:sus_bis} does not propose an explicit construction of the control template, but rather state that almost all choice of template must be good. In applications, one can either apply this reasoning and choose an arbitrary template, or propose an ad-hoc analysis on the system to construct a specific control template and apply Theorem \ref{th:main} with this template.

Combining Theorem \ref{th:main} and \ref{lem:sus_bis}, we therefore obtain a generic nonlinear separation principle based on the use of control templates, under the assumptions of exponential stabilizability and strong differential observability at the target.

\section{Output feedback stabilization theorem}
\label{sec:thmain}

This section is devoted to the proof of Theorem \ref{th:main}.
We first ensure well-posedness of the closed-loop hybrid system \eqref{eq:closed-loop}. Then, we provide preliminary results concerning the high-gain observer convergence on the one hand and the templated state feedback stabilization procedure on the other hand. Finally, we combine these two results to prove the stability of the resulting templated output feedback procedure.

\subsection{Well-posedness}\label{sec:wp}

We use the framework of hybrid systems developed in \cite{goedel2012hybrid} to define solutions of \eqref{eq:closed-loop}.
Note that \eqref{eq:closed-loop} clearly satisfies the hybrid basic conditions \cite[Assumption 6.5]{goedel2012hybrid}.
Moreover, the jump times of \eqref{eq:closed-loop} are determined by the autonomous hybrid subdynamics
\begin{equation}\label{eq:s}
\begin{cases}
    \dot s = 1 & 
    s\in[0, \dt],
    \\
    s^+ = 0 & s = \dt.
\end{cases}
\end{equation}
Hence, the jump times $(\tau_i)_{i\in\N}$ are given by $\tau_i = \dt-s_0+i\dt$ for $i\in\N$.
Thus, any solution $(x,z,s,\mu,\RR):E\to\R^n\times\R^{m(q+1)}\times[0, \dt]\times\R_+\times\O(p)$ to \eqref{eq:closed-loop} is a hybrid arc defined on a hybrid time domain $E\subset\R_+\times\N$
of the form
$E = \cup_{i=0}^{I-1}([0, T_e)\cap[\tau_i, \tau_{i+1}], i)$
where either $T_e=+\infty$ and $I=+\infty$ (complete trajectory) or
$T_e\in\R_+^*$, $I\in\N^*$, and $\tau_{I-1} < T_e \leq \tau_{I}$ (non-complete trajectory).
Since the flow map of \eqref{eq:closed-loop} is singled-valued and locally Lipschitz continuous, the Cauchy problem associated to the flow map admits a unique maximal solution of class $C^1$.
Thus,
for all parameters $\dt>0$, $(c_i)_{0\leq i\leq q}\in\R^{q+1}$ and $\theta>0$,
and
for each initial condition $(x_0,z_0,s_0,\mu_0,\RR_0)\in\R^n\times\R^{m(q+1)}\times[0, \dt]\times\R_+\times\O(p)$,
\eqref{eq:closed-loop}
admits a unique maximal solution 
$(x,z,s,\mu,\RR):E\to\R^n\times\R^{m(q+1)}\times[0, \dt]\times\R_+\times\O(p)$
such that
$t\mapsto(x,z,s,\mu,\RR)(t, i)$ is $C^1$ for each $i$.
Moreover, if $(x,z,s,\mu,\RR)$ remains bounded, then the trajectory is complete, i.e. $E = \cup_{i\in\N}([\tau_i, \tau_{i+1}], i)$.
In order to shorten the notations, we shall write
$(x,z,s,\mu,\RR)(t):=(x,z,s,\mu,\RR)(t, i)$ for all $(t,i)\in E$ such that $\tau_i < t \leq \tau_{i+1}$
and
$(x,z,s,\mu,\RR)^+(\tau_{i}):=(x,z,s,\mu,\RR)(\tau_{i}, i)$.

\subsection{High-gain observer with control templates}

In view of the definition of control templates in combination with Lemma \ref{Ass_HG}, the following Lemma follows directly from the usual results on high-gain observers (see e.g., \cite[Theorem 4.1]{bernard2019observer}, \cite[Theorem 6.1]{BERNARD2022224}, \cite[Chapter 6.2]{Gauthier_book}, \cite{256352}).
It will be used in the proof of Theorem \ref{th:main} to tune the convergence speed of the observer by picking $\theta$ sufficiently large.

\begin{lemma}[\hspace{-0.0001em}\protect{\cite[Theorem 4.1]{bernard2019observer}}]
\label{lem:hg}
Let $v^*\in C^\infty([0, T], \R^p)$ be a control template of order $q\in\N$.
For all $(\mu, \RR)\in [0, \bar\lambda]\times\O(p)$,
and all $(c_i)_{0\leq i\leq q}\in\R^{q+1}$ that are coefficient of a Hurwitz polynomial,
there exists $\theta^*_1\geq1$, $c\geq1$ and $\omega>0$ such that for all $\theta>\theta^*_1$,
the solutions of
\begin{equation}\label{obs:hg}
\left\{\begin{aligned}
&\dot x = f(x, \mu\RR v^*(t))
\\
&\dot z_0
= z_1 + \theta c_0(h(x)-z_0)\\
&\dot z_1 = z_2 + \theta^2 c_1(h(x)-z_0) \\
&\ \ \vdots\\
&\dot z_{q-1} = z_{q} + \theta^{q} c_{q-1}(h(x)-z_0) \\
&\dot z_{q} = H_{q+1}(\sat(\phi(z, s, \mu, \RR)),\mu\RR v^*(s),\dots,\mu\RR(v^*)^{(q)}(s)) + \theta^{q+1} c_{q}(h(x)-z_0)
\end{aligned}\right.
\end{equation}
are such that, for all $t\in[0, T]$,
\begin{equation}
    |e(t)|\leq c \theta^{q} \e^{-\theta \omega t} |e(0)|,
\end{equation}
where $e(t) := z(t) - \HH_q(x(t), t, \mu\RR v^*)$.
\end{lemma}

\subsection{Stabilization by means of a templated state feedback}
\label{S:stabilisation_templated}

In order to show that the observer-based dynamic output feedback stabilization strategy proposed in \eqref{eq:closed-loop} works, we need to show that it works in particular if the observer has converged, i.e. with a state-feedback.
When $v^*$ is constantly equal to $(1,0,\dots,0)$, the remaining closed-loop can be rewritten as a the following system with ``sample-and-hold'' input:
\begin{equation}\label{eq:closed-loop-state_const}
\left\{\begin{aligned}
&\left.\begin{aligned}
&\dot x = f(x, v)\hspace{1.25cm}
\\
&\dot s = 1
,\quad \dot v = 0
\end{aligned}\right\}
& s\in[0, \Delta],
\\
\\
&\left.\begin{aligned}
&x^+ = x
\\
&s^+ = 0
\\
&v^+ = \lambda(\sat(x))
\end{aligned}\qquad\right\}
& s = \Delta,
\end{aligned}\right.
\end{equation}
It is well-known that under Assumption \ref{ass:GES}, system \ref{eq:closed-loop-state_const} can be made LES with an arbitrarily large basin of attraction within $\XX$ by taking $\dt$ sufficiently small. 
For this reason (i.e., because we need robustness of the feedback law with respect to sample-and-hold inputs),
we require LES in Assumption \ref{ass:GES}, while usual nonlinear separation principles \cite{TeelPraly1994, TeelPraly1995, jouan, jouan1996finite} simply require LAS.
In the case where $v^*$ is non-constant, we need the following result to extend this robustness property in our context. Moreover, elements of its proof will be used {\it mutatis mutandis} in the proof of Theorem \ref{th:main}.

\begin{proposition}\label{lem:les}
    Under Assumption \ref{ass:GES},
    for any $T>0$, for any input $v^*\in C^1([0, T], \R^m)$ such that $v^*(0)=(1,0,\dots,0)$,
    there exists $\dt^*_0\in(0, T]$ such that, for all $\dt\in(0, \dt^*_0)$,
    the set $\{0\}\times[0, \dt]\times\{0\}\times\O(p)$ is LES for the closed-loop system
\begin{equation}\label{eq:closed-loop-state}
\left\{\begin{aligned}
&\left.\begin{aligned}
&\dot x = f(x, \mu \RR v^*(s))\qquad
\\
&\dot s = 1
,\quad\dot \mu = 0
,\quad\dot \RR = 0\hspace{0.45cm}
\end{aligned}\right\}
& s\in[0, \Delta],
\\
\\
&\left.\begin{aligned}
&x^+ = x
\\
&s^+ = 0
\\
&\mu^+ = |\lambda(\sat(x))|
\\
&\RR^+ \in \RRR_0(\lambda(\sat(x)))\hspace{1cm}
\end{aligned}\right\}
& s = \Delta,
\end{aligned}\right.
\end{equation}
with basin of attraction containing $\KK_x\times[0, \dt]\times[0, \bar\lambda]\times\O(p)$.
\end{proposition}

We postpone the proof after two preliminary lemmas, but
regarding the existence of solutions to \eqref{eq:closed-loop-state}, we follow the same method as in Section \ref{sec:wp}. For all $\dt>0$
and
for each initial condition $(x_0,s_0,\mu_0,\RR_0)\in\R^n\times[0, \dt]\times\R_+\times\O(p)$,
\eqref{eq:closed-loop-state}
admits a unique maximal solution 
$(x,s,\mu,\RR):E\to\R^n\times[0, \dt]\times\R_+\times\O(p)$
such that
$t\mapsto(x,s,\mu,\RR)(t, i)$ is $C^1$ for each $i$.
Moreover, if $(x,s,\mu,\RR)$ remains bounded, then the trajectory is complete, i.e. $E = \cup_{i\in\N}([\tau_i, \tau_{i+1}], i)$.
In what follows, for a given maximal solution $x$ of \eqref{eq:closed-loop-state} $t^*$ denotes the escape time out of the compact set $\KK_x'$, $t^* = \inf\{t>0\mid x(t)\notin\KK_x'\}$ (with $\KK_x'$ as in Section~\ref{S:Hybrid dynamic output feedback design}).

According the converse Lyapunov function theorem \cite[Remark 2.350]{PralyBresch2022a},
there exist a Lyapunov function $V$ and four positive constants $(\alpha_i)_{1\leq i\leq 4}$ such that for all $\xi\in\KK_x'$, \eqref{eq:Vexp} is satisfied.
Let us give a first bound on the growth of the Lyapunov $V$ on the interval $[0,\tau_0)$, before the feedback law kicks in.

\begin{lemma}\label{L:beginning}
Under the assumptions of Proposition~\ref{lem:les}, assume $\tau_0>0$.
For all $t\in(0, \min(\tau_0, t^*))$.
\begin{align}\label{eq:V7bis}
    V(x(t))\leq \e^{\dt\frac{\alpha_3\lip_f}{\alpha_1}}\Big(\sqrt{V(x_0)}
    +\dt \frac{\lip_f \|v^*\|_{\infty}}{2\sqrt{\alpha_1}}
    |\mu_0|
    \Big)^2
\end{align}
with  $\lip_f$ denoting the Lipschitz constant of $(x,u)\mapsto f(x, u)$ over $\KK_x'\times\bar B_{\R^p}(0, \bar\lambda\|v^*\|_{\infty})$.
\end{lemma}

\begin{proof}
This result is an application of nonlinear Grönwall's inequality \cite{perov1959k} (see also \cite[Theorem 21]{dragomir2003some}).
First, for all $t\in(0, \min(\tau_0, t^*))$,
\begin{align}
    V(x(t))
    &\leq V(x_0) + \int_0^t \left|\frac{\partial V}{\partial x}(x(\tau))f(x(\tau), \mu_0\RR_0v^*(s(\tau))\right|\dd\tau
    \nonumber\\
    &\leq V(x_0) + \alpha_3\lip_f \int_0^t |x(\tau)|^2 \dd\tau + \lip_f \|v^*\|_{\infty} |\mu_0|  \int_0^t |x(\tau)| \dd\tau
    \nonumber\\
    &\leq V(x_0) + \frac{\alpha_3\lip_f}{\alpha_1} \int_0^t V(x(\tau)) \dd\tau + \frac{\lip_f \|v^*\|_{\infty}}{\sqrt{\alpha_1}}|\mu_0|  \int_0^t \sqrt{V(x(\tau)))} \dd\tau
    \label{eq:Vr}
\end{align}
Hence, according to the nonlinear Grönwall's inequality,
\[
    V(x(t))\leq \e^{t\frac{\alpha_3\lip_f}{\alpha_1}}\Big(\sqrt{V(x_0)}
    +t \frac{\lip_f \|v^*\|_{\infty}}{2\sqrt{\alpha_1}}
    |\mu_0|
    \Big)^2
\]
for all $t\in(0, \min(\tau_0, t^*))$. Hence the statement since $\tau_0\in [0,\Delta]$.
\end{proof}

Now we provide a natural intermediary result on templated state feedback, namely that after $\tau_0$, we can bound the gap between the state feedback and templated feedback at any time in the window $[\tau_i,\tau_{i+1})$ by reducing $\dt$. We also wish to relate this bound to the size of the state, in order to relate it to the Lyapunov function $V$. What is crucial here is that rather than keeping memory of $|x(\tau_i)|$, we can actually recover $|x(t)|$.
Let us introduce Lipschitz constants that appear in the proof of this fact. Here,
$\lip_{\lambda}$, $\lip_{v^*}$, $\lip_{\sat}$, and $\lip_{f_\lambda}$ denote the Lipschitz constants of $\lambda|_{\KK_x'}$, $v^*$, $\sat|_{\KK_x'}$, and $f(\cdot, \lambda(\cdot))|_{\KK_x'}$, respectively.

\begin{lemma}\label{L:feedbackdrift}
Under the assumptions of Proposition~\ref{lem:les},
there exists $\dt_1^*\in (0,T]$ and $\alpha_{5}:(0, \dt_1^*)\to\R_+^*$ such that $\alpha_{5}(\dt)\to0$ as $\dt\to0$ and if $\dt\in(0, \dt^*_1)$, then
for all $(t, i)\in E$ such that $\tau_i<t< \tau_{i+1}$ and $t<t^*$,
\begin{equation}\label{eq:VfinlemLud}
    |\mu(t)\RR(t)v^*(s(t)) - \lambda(x(t))|
    \leq \alpha_{5}(\dt)|x(t)|.
\end{equation}
\end{lemma}
\begin{proof}
Using Lemma \ref{lem:R0}, for all $t\in(\tau_0, t^*)$, pick $\tilde\RR(\lambda(x(t)))\in \RRR_0(\lambda(x(t)))$, so that
$\||\lambda(x(t))|\tilde\RR(\lambda(x(t)))-\mu(t)\RR(t)\|\leq 2|\lambda(x(t)) - \mu(t)|$.
Then, for all $(t, i)\in E$ such that $\tau_i<t< \tau_{i+1}$ and $t<t^*$,
\begin{align}
    |\mu\RR v^*(s) - \lambda(x)|
    &\leq \big|\mu \RR v^*(s ) - |\lambda(x )|\tilde\RR(\lambda(x ))v^*(s )\big| 
    + \big| |\lambda(x )|\tilde\RR(\lambda(x ))v^*(s ) - \lambda(x )\big|
    \nonumber\\
    &\leq |v^*(s )| \, \big\|\mu \RR  - |\lambda(x )|\tilde\RR(\lambda(x ))\big\|  
    + |\lambda(x )| 
    \,
    |v^*(s )-v^*(0)|
    \nonumber\\
    &\leq \|v^*\|_{\infty} \big|\mu  - |\lambda(x )|\big|   
    + \lip_{\lambda}\lip_{v^*}\dt|x |
    \nonumber\\
    &\leq \|v^*\|_{\infty} \lip_{\lambda}\lip_{\sat} |x(\tau_i)-x |
    + \lip_{\lambda}\lip_{v^*}\dt|x |
    \label{eq:V22}
\end{align}
Moreover,
\begin{align}
    \sup_{\sigma\in[\tau_i, t]} |x(t)-x(\sigma)|
    &\leq \int_{\tau_i}^t |f(x(\tau), \mu(\tau)\RR(\tau)v^*(s(\tau)))|\dd\tau
    \nonumber\\
    &\leq \int_{\tau_i}^t |f(x(\tau), \lambda(x(\tau)))|\dd\tau
    \nonumber\\
    &\quad+ \int_{\tau_i}^t |f(x(\tau), \mu(\tau)\RR(\tau)v^*(s(\tau))) - f(x(\tau), \lambda(x(\tau)))|\dd\tau
    \nonumber\\
    &\leq \dt \lip_{f_\lambda} \sup_{\tau\in[\tau_i, t]} |x(\tau)|
    + \lip_f  \int_{\tau_i}^t  |\mu(\tau)\RR(\tau)v^*(s(\tau)) - \lambda(x(\tau))| \dd\tau
    \nonumber\\
    &\leq \dt \lip_{f_\lambda} \sup_{\tau\in[\tau_i, t]} |x(\tau)|
    + \lip_f   \|v^*\|_{\infty} \lip_{\lambda}\lip_{\sat}\int_{\tau_i}^t  |x(\tau_i)-x(\tau)| \dd\tau
    \nonumber\\
    &\quad+\dt \lip_f \lip_{\lambda}\lip_{v^*} \sup_{\tau\in[\tau_i, t]} |x(\tau)|
    \label{eq:V3}
\end{align}
where \eqref{eq:V22} has been used to obtain the last inequality.
Hence, according to Grönwall's inequality:
\begin{equation}\label{eq:V3bis}
\sup_{\sigma\in[\tau_i, t]} |x(t)-x(\sigma)|
\leq \e^{\dt \lip_f   \|v^*\|_{\infty} \lip_{\lambda}\lip_{\sat} }\dt(\lip_{f_\lambda}+ \lip_f\lip_{\lambda}\lip_{v^*}) \sup_{\tau\in[\tau_i, t]} |x(\tau)|
\end{equation}
Thus,
\begin{align*}
\sup_{\tau\in[\tau_i, t]} |x(\tau)|
&\leq 
|x(t)| + \sup_{\sigma\in[\tau_i, t]} |x(t)-x(\sigma)|
\nonumber\\
&\leq |x(t)| + \e^{\dt \lip_f   \|v^*\|_{\infty} \lip_{\lambda}\lip_{\sat} }\dt(\lip_{f_\lambda}+ \lip_f\lip_{\lambda}\lip_{v^*}) \sup_{\tau\in[\tau_i, t]} |x(\tau)|,
\end{align*}
i.e., if 
$\dt$ is sufficiently small so that
$\alpha_{6}(\dt) = (1-\e^{\dt \lip_f  \|v^*\|_{\infty} \lip_{\lambda}\lip_{\sat} }\dt(\lip_{f_\lambda}+ \lip_f\lip_{\lambda}\lip_{v^*}))^{-1}$ is well-defined and positive, we have
\begin{align}\label{eq:V5}
    \sup_{\tau\in[\tau_i, t]} |x(\tau)| \leq \alpha_{6}(\dt) |x(t)|.
\end{align}
Combining \eqref{eq:V22}, \eqref{eq:V3bis} and \eqref{eq:V5}, we obtain
\begin{align*}
    |\mu(t)\RR(t)v^*(s(t)) - \lambda(x(t))|
    &\leq \|v^*\|_{\infty} \lip_{\lambda}\lip_{\sat}|x(\tau_i)-x(t)|
    + \lip_{\lambda}\lip_{v^*}\dt|x(t)|
    \nonumber\\
    &\leq \alpha_{5}(\dt)|x(t)|,
\end{align*}
with $\alpha_{5}(\dt) = \|v^*\|_{\infty} \lip_{\lambda}\lip_{\sat} \e^{\dt \lip_f  \|v^*\|_{\infty} \lip_{\lambda}\lip_{\sat}}\dt(\lip_{f_\lambda}+ \lip_f\lip_{\lambda}\lip_{v^*})\alpha_{6}(\dt)
    + \lip_{\lambda}\lip_{v^*}\dt$.
Note that $\alpha_{6}(\dt)\to1$ and then $\alpha_{5}(\dt)\to0$ as $\dt\to0$.
\end{proof}

We are now ready to prove the main result of the section.

\begin{proof}[Proof of Proposition~\ref{lem:les}]
Since $\lambda$ is Lipschitz continuous over $\KK_x$, it is sufficient to show that there exist $M\geq1$ and $\nu>0$ such that, for all initial condition $(x_0, s_0, \mu_0, \RR_0)\in \KK_x\times[0, \dt]\times[0, \bar\lambda]\times\O(p)$, the corresponding trajectory $(x, s, \mu, \RR)$ is such that, for all $t\geq0$,
\begin{equation}\label{eq:Vf}
    |x(t)|\leq M \e^{-\nu t} (|x_0| + |\mu_0|).
\end{equation}
Then, for almost all $t\in(0, t^*)$,
\begin{align}
    \frac{\dd}{\dd t}V(x  )
    &= \frac{\partial V}{\partial x}(x  )f(x  , \mu  \RR  v^*(s  ))\nonumber
    \\
    &= \frac{\partial V}{\partial x}(x  )\big[f(x  , \lambda(x  ))+ f(x  , \mu  \RR  v^*(s  )) - f(x  , \lambda(x  )) \big]\nonumber
    \\
    &\leq -\alpha_4|x  |^2 + \alpha_3\lip_{f}|x  ||\mu  \RR  v^*(s  ) - \lambda(x  )|.\label{eq:V1}
\end{align}

From Lemma~\ref{L:feedbackdrift}, there exists $\dt_2^*>0$ such that, if $\dt\in(0, \dt_2^*)$, then 
$\alpha_3\alpha_{5}(\dt)\lip_{f}<\alpha_4$.
Thus, with \eqref{eq:V1}, we obtain that for almost all $t\in(\tau_0,t^*)$,
\begin{align*}
    \frac{\dd}{\dd t}V(x(t))
    \leq -(\alpha_4-\alpha_3\alpha_{5}(\dt)\lip_{f})|x(t)|^2\leq -\frac{\alpha_4-\alpha_3\alpha_{5}(\dt)\lip_{f}}{\alpha_1}V(x(t)),
\end{align*}
hence, since $\tau_0<\dt$,
\begin{align}\label{eq:V6}
    V(x(t))
    \leq \e^{-(t-\dt)\frac{\alpha_4-\alpha_3\alpha_{5}(\dt)\lip_{f}}{\alpha_1}}V(x(\tau_0)).
\end{align}
Now from Lemma~\ref{L:beginning}, we obtain that for all $t\in[0, t^*)$,
\begin{equation}
 V(x(t)) \leq \e^{-(t-\dt)\frac{\alpha_4-\alpha_3\alpha_{5}(\dt)\lip_{f}}{\alpha_1}} \e^{\dt\frac{\alpha_3\lip_f}{\alpha_1}}\Big(\sqrt{V(x_0)}
    +\dt \frac{\lip_f \|v^*\|_{\infty}}{2\sqrt{\alpha_1}}
    |\mu_0|
    \Big)^2
\end{equation}
Let $r\in(0, \bar r)$ be such that $V(\xi)\leq r$ for all $\xi\in\KK_x$.
Let $\dt^*_3>0$ be such that, for all $\dt\in(0, \dt^*_3)$,
\begin{equation}\label{eq:rbar}
 \e^{\dt\frac{\alpha_4-\alpha_3\alpha_{5}(\dt)\lip_{f}}{\alpha_1}} \e^{\dt\frac{\alpha_3\lip_f}{\alpha_1}}\Big(\sqrt{r}
    +\dt \frac{\lip_f \|v^*\|_{\infty}}{2\sqrt{\alpha_1}}
    \bar\lambda
    \Big)^2
    < \bar r.
\end{equation}
Set $\dt^*_0 = \min(\dt^*_2, \dt^*_3)$.
Then, if $\dt\in(0, \dt^*)$, $x$ remains in $\KK_x'$ (i.e. $V(x(t))\leq \bar r$ for all $t\in\R_+$, i.e. $t^*=+\infty)$,
and
\eqref{eq:Vf} is satisfied with $\nu = \frac{\alpha_4-\alpha_3\alpha_{5}(\dt)\lip_{f}}{\alpha_1}$ and
$M = \frac{1}{\sqrt{\alpha_1}}\e^{\dt\nu}\max(\e^{\dt\frac{\alpha_3\lip_f}{2\alpha_1}}, \dt \frac{\lip_f \|v^*\|_{\infty}}{2\sqrt{\alpha_1}}
    \e^{\dt\frac{\alpha_3\lip_f}{2\alpha_1}})$.
\end{proof}

\subsection{Stabilization by means of templated output feedback}

In this subsection, we prove Theorem \ref{th:main}.
We show that there exists
$C\geq1$ and $\varpi>0$ such that,
for any $(x_0,z_0,s_0,\mu_0,\RR_0)\in\KK_x\times\KK_z\times[0, \dt]\times\R_+\times\O(p)$,
the corresponding unique maximal solution
$(x,z,s,\mu,\RR):E\to\R^n\times\R^{m(q+1)}\times[0, \dt]\times\R_+\times\O(p)$
of \eqref{eq:closed-loop}
is complete and
\begin{equation*}
    \|(x,z,\mu)(t)\| \leq C\e^{-\varpi t}\|(x_0,z_0,\mu_0)\|.
\end{equation*}
For one such given maximal solution, denote
by $T_e\in\R_+^*\cup\{+\infty\}$ the time of existence of the solution.
Clearly, $s$, $\mu$ and $\RR$ remain in the bounded sets $[0, \dt]$, $[0, \bar\lambda]$ and $\O(p)$, respectively.
Hence, we show that $x$ and $z$ remain bounded, which implies $T_e = +\infty$.
Let $t^* = \inf\{t\in[0, T_e)\mid x(t)\notin\KK_x'\}$ (with the convention that $t^*=T_e$ when the set is empty).
In particular, we show that $t^*=T_e=+\infty$, i.e. $x$ remains in $\KK_x'$ and $z$ is bounded.
For all $(t,i)\in E$, define $e(t, i)= z(t, i) - \HH_q(x(t, i), s(t, i), \mu(t, i)\RR(t, i) v^*)$ for all $(t, i)\in E$.

We first investigate the exponential stability of the estimation error. In what follows, $\lip_{\HH_q}$ denote the Lipschitz constant of $(x, \mu, \RR)\mapsto\HH_q(x, 0, \mu\RR v^*)$ over the compact set $\KK_x'\times [0, \bar\lambda]\times\O(p)$.

\begin{lemma}\label{lem:exp_e}
Let $\Delta\in (0,T]$. Under the assumptions of Theorem~\ref{th:main},
there exists $\theta_2^*(\Delta)\geq\theta_1^*$ (see Lemma~\ref{lem:hg}) such that for any choice of $\theta\geq\theta_2^*$,
    there exists $\bar c(\theta, \dt)\geq1$ and $\bar \omega(\theta)>0$ such that, for all $(t, i)\in E$ such that $t<t^*$,
    \begin{equation}\label{eq:exp_e}
        |e(t, i)|\leq \bar c \e^{-\bar \omega t} |(x_0, z_0, \mu_0)|
    \end{equation}
\end{lemma}

\begin{proof}

On the one hand, for all $i\in\N$ such that $\tau_i < t^*$, {since $z^+=\HH_q(\sat(\phi(z, \Delta, \mu, \RR)), 0, \mu^+)\RR^+ v^*) $,}
\begin{align*}
|e^+(\tau_i)|
&= |z^+(\tau_i) - \HH_q(x^+(\tau_i), s^+(\tau_i), \mu^+(\tau_i)\RR^+(\tau_i) v^*)|
\\
&\leq \lip_{\HH_q} \big|\sat\big(\phi(z(\tau_i), \Delta, \mu(\tau_i), \RR(\tau_i))\big)-x(\tau_i)\big|.
\end{align*}
With $x(\tau_i)=\sat(x(\tau_i))=\sat\big(\phi(\HH_q(x(\tau_i), \dt, \mu(\tau_i)\RR(\tau_i) v^*), \dt, \mu(\tau_i), \RR(\tau_i))\big)$, 
\[
|e^+(\tau_i)|
\leq \lip_{\HH_q} \lip_{\sat}\lip_{\phi} |z(\tau_i) - \HH_q(x(\tau_i), \dt,\mu(\tau_i)\RR(\tau_i)v^*)|
= \lip_{\HH_q} \lip_{\sat}\lip_{\phi} |e(\tau_i)|,
\]
On the other hand, if $\theta\geq \theta_1^*$, we apply Lemma \ref{lem:hg}  on $(\tau_i, \tau_{i+1})$ for all $(t, i)\in E$ such that $t<t^*$ and we get for $\tau_i < t \leq \tau_{i+1}$:
\begin{equation}
    |e(t)| \leq c \theta^{q} \e^{-\theta \omega (t-\tau_i)} |e^+(\tau_i)|.
\end{equation}
An immediate induction shows that for all $t\in(\tau_0, t^*)$ and $i\in\N$ such that $\tau_i < t \leq \tau_{i+1}$,
\begin{equation*}
    |e(t)|
    \leq \e^{-\theta \omega (t-\tau_i)} \e^{-\theta \omega i\dt} \Big(c \theta^{q} \lip_{\HH_q} \lip_{\sat}\lip_{\phi} \Big)^{i+1} |e(\tau_0)|.
\end{equation*}
If $s_0=\dt$, then $\tau_0 = 0$. Otherwise, $|e(\tau_0)|\leq c \theta^{q} \e^{-\theta \omega \tau_0} |e(0)|$.
In any case,
\[
|e(t)|
    \leq \e^{-\theta \omega (t-\tau_i)} \e^{-\theta \omega i\dt} (c \theta^{q} \lip_{\HH_q} \lip_{\sat}\lip_{\phi} )^{i+1} c \theta^{q} |e(0)|
\]
Now let $\theta_2^*(\Delta)\geq \theta_1^*$ be such that if $\theta\geq \theta_2^*$, then $\e^{\frac{\theta \omega \dt}{2}}     
    \geq
    c \theta^{q} \lip_{\HH_q} \lip_{\sat}\lip_{\phi}$. As a result, with $\theta\geq \theta_2^*$
\begin{align}
    |e(t)|
    &\leq \e^{-\theta \omega (t-\tau_i)} \e^{-\frac{\theta \omega i\dt}{2}} (c \theta^{q})^2 \lip_{\HH_q} \lip_{\sat}\lip_{\phi} |e(0)|
    \nonumber\\
    &\leq \e^{-\frac{\theta \omega }{2}(t-\tau_0)} (c \theta^{q})^2 \lip_{\HH_q} \lip_{\sat}\lip_{\phi} |e(0)|
    \label{eq:Vebis}
    \\
    &\leq \e^{-\frac{\theta \omega }{2}t} \e^{\frac{\theta \omega }{2}\dt} (c \theta^{q})^2 \lip_{\HH_q} \lip_{\sat}\lip_{\phi} |e(0)|
    \label{eq:Ve}
\end{align}
Moreover, using $\HH_q(0, s_0, 0)=0$, we have by triangular inequality
\[
    |e(0)|
    \leq |z_0| + |\HH_q(x_0, s_0, \mu_0\RR_0 v^*) - \HH_q(0, s_0, 0)|
\leq |z_0| + \lip_{\HH_q}(|x_0| + \|v^*\|_\infty |\mu_0|).
\]
Hence, for all $(t, i)\in E$ such that $t<t^*$,
\[
    |e(t, i)|
    \leq \e^{-\frac{\theta \omega }{2}t} \e^{\frac{\theta \omega }{2}\dt} (c \theta^{q})^2 \lip_{\HH_q} \lip_{\sat}\lip_{\phi} (|z_0| + \lip_{\HH_q}(|x_0| + \|v^*\|_\infty |\mu_0|)).
\]
Hence, $e$ remains bounded over $[0, t^*)$, and \eqref{eq:exp_e} is satisfied for a suitable choice of $(\bar c, \bar \omega)$.
\end{proof}

Now, we investigate the exponential stability of the system's state.

First we give an estimate of  the gap between the state feedback and templated output feedback, in the same fashion as we did for the templated feedback in Lemma~\ref{L:feedbackdrift}, except now an estimation error is bound to appear. Some elements of the proof are reminiscent of the proof of Lemma~\ref{L:feedbackdrift}, so they may be exposed more succinctly.

\begin{lemma}\label{L:gap_output_feedback}
Under the assumptions of Theorem~\ref{th:main},
there exists $\alpha_{7}(\dt)\in \R^+$ such that if $\dt\in(0, \dt^*_1)$, then
for all $(t, i)\in E$ such that $\tau_i<t< \tau_{i+1}$ and $t<t^*$,
\begin{equation}
    |\mu(t)\RR(t)v^*(s(t)) - \lambda(x(t))|
    \leq \alpha_{5}(\dt)|x(t)| + \alpha_{7}(\dt)|e(\tau_i)|.
\end{equation}
\end{lemma}

\begin{proof}
    Using Lemma \ref{lem:R0}, for all $i\in\N$ such that $\tau_i<t^*$, pick $\bar\RR(\lambda(x(\tau_i)))\in \RRR_0(\lambda(x(\tau_i)))$, so that
$\||\lambda(x(\tau_i))|\bar\RR(\lambda(x(\tau_i)))-\mu(t)\RR(t)\|\leq 2|\lambda(x(\tau_i)) - \lambda(\sat(\phi(z(\tau_i), \dt, \mu(\tau_i))))|$.
Then, for all $i\in\N$ and all $t\in(\tau_i, \tau_{i+1}]$ such that $t<t^*$,
\begin{align}
    |\mu(t)\RR(t)v^*(s(t)) - \lambda(x(t))|
    &\leq |\mu(t)\RR(t)v^*(s(t)) - |\lambda(x(\tau_i))|\bar\RR(\lambda(x(\tau_i)))v^*(s(t))|
    \nonumber\\
    &\quad+ ||\lambda(x(\tau_i))|\bar\RR(\lambda(x(\tau_i)))v^*(s(t)) - \lambda(x(t))|.
    \label{eq:V8}
\end{align}
On the one hand,
\begin{align*}
    |\mu(t)\RR(t)v^*(s(t)) &- |\lambda(x(\tau_i))|\bar\RR(\lambda(x(\tau_i)))v^*(s(t))|
    \leq \|v^*\|_{\infty} \|\mu(t)\RR(t) - |\lambda(x(\tau_i))|\bar\RR(\lambda(x(\tau_i)))\|
    \\
    &\leq \|v^*\|_{\infty} 
    |\mu(t) - \lambda(x(\tau_i))|
    \\
    &\leq \|v^*\|_{\infty} \lip_\lambda 
    |\sat(\phi(z(\tau_i), \dt, \mu(\tau_i))) - x(\tau_i)|
\end{align*}
since $\mu(t) = |\lambda(\sat(\phi(z(\tau_i), \dt, \mu(\tau_i))))|$.
Since, for all $t\in[0, t^*)$, $x(t) = \sat(x(t))$ and $$x(t) = \phi(\HH_q(x(t), s(t), \mu(t)\RR(t) v^*), s(t), \mu(t), \RR(t))$$ by definition of $\phi$, we obtain that
\begin{align}
    |\mu(t)\RR(t)v^*(s(t)) &- |\lambda(x(\tau_i))|\bar\RR(\lambda(x(\tau_i)))v^*(s(t))|
    \nonumber\\
    &\leq \|v^*\|_{\infty} \lip_\lambda \lip_{\sat}\lip_\phi|z(\tau_i) - \HH_q(x(\tau_i), \dt, \mu(\tau_i)\RR(\tau_i) v^*)|
    \nonumber\\
    &\leq \|v^*\|_{\infty} \lip_\lambda \lip_{\sat}\lip_\phi|e(\tau_i)|.
    \label{eq:V9}
\end{align}
On the other hand, reasoning as in Lemma~\ref{L:feedbackdrift} we have
\begin{align}
    ||\lambda(x(\tau_i))|\bar\RR(\lambda(x(\tau_i)))v^*(s(t)) - \lambda(x(t))|
    \leq \|v^*\|_{\infty} \lip_{\lambda}\lip_{\sat} |x(\tau_i)-x(t)|
    + \lip_{\lambda}\lip_{v^*}\dt|x(t)|
    \label{eq:V10}
\end{align}
Moreover,
\begin{align}
    \sup_{\sigma\in[\tau_i, t]} |x(t)-x(\sigma)|
    &\leq \int_{\tau_i}^t |f(x(\tau), \mu(\tau)\RR(\tau)v^*(s(\tau)))|\dd\tau
    \nonumber\\
    &\leq \int_{\tau_i}^t |f(x(\tau), |\lambda(x(\tau_i))|\bar\RR(\lambda(x(\tau_i)))v^*(s(\tau)))|\dd\tau
    \nonumber\\
    &\quad + \int_{\tau_i}^t |f(x(\tau), \mu(\tau)\RR(\tau)v^*(s(\tau))) - f(x(\tau), |\lambda(x(\tau_i))|\bar\RR(\lambda(x(\tau_i)))v^*(s(\tau)))|\dd\tau
    \label{eq:V11}
\end{align}
On one hand, reasoning as in \eqref{eq:V3}, we get
\begin{align}
    \int_{\tau_i}^t |f(x(\tau), |\lambda(x(\tau_i))|\bar\RR(\lambda(x(\tau_i)))v^*(s(\tau)))|\dd\tau
    &\leq \dt \lip_{f_\lambda} \sup_{\tau\in[\tau_i, t]} |x(\tau)|
    \nonumber\\
    &\quad+ \lip_f  \|v^*\|_{\infty} \lip_{\lambda}\lip_{\sat}\int_{\tau_i}^t  |x(\tau_i)-x(\tau)| \dd\tau
    \nonumber\\
    &\quad+\dt \lip_f \lip_{\lambda}\lip_{v^*} \sup_{\tau\in[\tau_i, t]} |x(\tau)|.\label{eq:V12}
\end{align}
On the other hand,
\begin{align}
    \int_{\tau_i}^t |f(x(\tau), \mu(\tau)\RR(\tau)v^*(s(\tau))) &- f(x(\tau), |\lambda(x(\tau_i))|\bar\RR(\lambda(x(\tau_i)))v^*(s(\tau)))|\dd\tau
    \nonumber\\
    &\leq \dt\lip_f\sup_{\tau\in[\tau_i, t]} |\mu(\tau)\RR(\tau)v^*(s(\tau)) - |\lambda(x(\tau_i))|\bar\RR(\lambda(x(\tau_i)))v^*(s(\tau))|
    \nonumber\\
    &\leq \dt\lip_f \|v^*\|_{\infty} \lip_\lambda \lip_{\sat}\lip_\phi|e(\tau_i)|.\label{eq:V13}
\end{align}
Combining \eqref{eq:V12} and \eqref{eq:V13} into \eqref{eq:V11} and using Grönwall's inequality, we obtain
\begin{align}
\sup_{\sigma\in[\tau_i, t]} |x(t)-x(\sigma)|
\leq \e^{\dt \lip_f  \|v^*\|_{\infty} \lip_{\lambda}\lip_{\sat} }
\Big(&\dt(\lip_{f_\lambda}+ \lip_f\lip_{\lambda}\lip_{v^*}) \sup_{\tau\in[\tau_i, t]} |x(\tau)|
\nonumber\\
&+\dt\lip_f \|v^*\|_{\infty} \lip_\lambda \lip_{\sat}\lip_\phi|e(\tau_i)|\Big)
\label{eq:V14}
\end{align}
Thus, similarly to \eqref{eq:V5}, we obtain
\begin{align}\label{eq:V15}
    \sup_{\tau\in[\tau_i, t]} |x(\tau)| \leq \alpha_{6}(\dt) |x(t)| + \alpha_{8}(\dt)|e(\tau_i)|
\end{align}
with $\alpha_{8}(\dt) = \e^{\dt \lip_f  \|v^*\|_{\infty} \lip_{\lambda}\lip_{\sat}}\dt\lip_f \|v^*\|_{\infty} \lip_\lambda \lip_{\sat}\lip_\phi$.
Combining \eqref{eq:V8}, \eqref{eq:V9}, \eqref{eq:V10}, \eqref{eq:V14} and \eqref{eq:V15} yields the stated bound
\begin{align}\nonumber
    |\mu(t)\RR(t)v^*(s(t)) - \lambda(x(t))|
    \leq \alpha_{5}(\dt)|x(t)| + \alpha_{7}(\dt)|e(\tau_i)|
\end{align}
with
\begin{align}
\alpha_{7}(\dt) &= 2\|v^*\|_{\infty} \lip_\lambda \lip_{\sat}\lip_\phi
+
(\|v^*\|_{\infty} \lip_{\lambda}\lip_{\sat})^2
\e^{\dt \lip_f  \|v^*\|_{\infty} \lip_{\lambda}\lip_{\sat}}
\dt\lip_f
\nonumber\\
&\quad+
\|v^*\|_{\infty} \lip_{\lambda}\lip_{\sat}
\e^{\dt \lip_f  \|v^*\|_{\infty} \lip_{\lambda}\lip_{\sat}}
\dt(\lip_{f_\lambda}+ \lip_f\lip_{\lambda}\lip_{v^*})
\alpha_{8}(\dt).
\label{eq:alpha8}
\end{align}
\end{proof}

Now that the gap in control is bounded, we can prove exponential stability of the state.

\begin{lemma}\label{lem:exp_x}
Under the assumptions of Theorem~\ref{th:main}, there exists $\Delta_4^*\in (0,T]$ such that for all $\Delta\in (0,\Delta_4^*]$, there exists $\theta_3^*(\Delta)\geq\theta_2^*(\Delta)$ (see Lemma~\ref{lem:exp_e}) such that for any choice of  $\theta\geq \theta_3^*$,
    we have $t^*=T_e$, and
    there exists $\bar M(\theta, \dt)\geq1$ and $\bar \nu(\theta, \dt)>0$ such that, for all $(t, i)\in E$ such that $t<T_e$,
    \begin{equation}\label{eq:exp_x}
        |x(t, i)|\leq \bar M \e^{-\bar \nu t} |(x_0, z_0, \mu_0)|
    \end{equation}
\end{lemma}

\begin{proof}

Reasoning as in the proof of Proposition~\ref{lem:les}, we have that for almost all $t\in(0, t^*)$,
\begin{align}
    \frac{\dd}{\dd t}V(x(t))
    \leq -\alpha_4|x(t)|^2 + \alpha_3\lip_{f}|x(t)||\mu(t)\RR(t)v^*(s(t)) - \lambda(x(t))|.\label{eq:V1e}
\end{align}
Assuming $\Delta<\Delta^*_1$, from Lemma~\ref{L:gap_output_feedback}, we obtain that for all $i\in\N$ and almost all $t\in(\tau_i,\tau_{i+1})$ such that $t<t^*$,
\begin{align*}
    \frac{\dd}{\dd t}V(x(t))
    &\leq -(\alpha_4-\alpha_3\alpha_{5}(\dt)\lip_{f})|x(t)|^2
    + \alpha_3\lip_f\alpha_{7}(\dt)|x(t)||e(\tau_i)|
    \\
    &\leq -(\alpha_4-\alpha_3\alpha_{5}(\dt)\lip_{f})|x(t)|^2
    + \alpha_3\lip_f\alpha_{7}(\dt)(\frac{\eta}{2}|x(t)|^2+\frac{1}{2\eta}|e(\tau_i)|^2)
\end{align*}
for all $\eta>0$, by Young's inequality. Picking $\eta = \frac{\alpha_4-\alpha_3\alpha_{5}(\dt)\lip_{f}}{\alpha_3\lip_f\alpha_{7}(\dt)}$ yields
\begin{align}
    \frac{\dd}{\dd t}V(x(t))
    &\leq -\frac{\alpha_4-\alpha_3\alpha_{5}(\dt)\lip_{f}}{2\alpha_1}V(x(t))
    + \frac{(\alpha_3\lip_f\alpha_{7}(\dt))^2}{2(\alpha_4-\alpha_3\alpha_{5}(\dt)\lip_{f})}|e(\tau_i)|^2.
    \label{eq:V30}
\end{align}
Let $r\in(0, \bar r)$ be such that $V(\xi)\leq r$ for all $\xi\in\KK_x$. Reasoning as in the proof of \eqref{eq:V7bis}, we obtain that for all $t\in[0, \min(\tau_1, t^*))$,
\begin{equation*}
 V(x(t))\leq \e^{2\dt\frac{\alpha_4-\alpha_3\alpha_{5}(\dt)\lip_{f}}{\alpha_1}} \e^{2\dt\frac{\alpha_3\lip_f}{\alpha_1}}\Big(\sqrt{r}
    +2\dt \frac{\lip_f \|v^*\|_{\infty}}{2\sqrt{\alpha_1}}
    \bar\lambda
    \Big)^2
\end{equation*}
Set $\Delta_4^*\in (0,T]$ to be such that if $\Delta\in (0,\Delta_4^*]$ then 
$
\e^{2\dt\frac{\alpha_4-\alpha_3\alpha_{5}(\dt)\lip_{f}}{\alpha_1}} \e^{2\dt\frac{\alpha_3\lip_f}{\alpha_1}}\Big(\sqrt{r}
    +2\dt \frac{\lip_f \|v^*\|_{\infty}}{2\sqrt{\alpha_1}}
    \bar\lambda
    \Big)^2
    <\bar r
$
(with $\bar r$ defined in Section~\ref{S:Hybrid dynamic output feedback design}). Then assuming that $\Delta\in (0,\Delta_4^*]$ yields that $V(x(\min(\tau_1,t^*)))<\bar r$, while by definition $V(x(t^*))=\bar r $ if $t^*<\infty$. Hence $t^*>\tau_1$. Assume that $t^*<T_e$. Then $V(x(t^*))=\bar r$ and $\frac{\dd}{\dd t}V(x(t^*)) \geq 0$. Equation \eqref{eq:V30} yields
\begin{align}\label{eq:Ve1}
    \frac{\alpha_4-\alpha_3\alpha_{5}(\dt)\lip_{f}}{2\alpha_1}\bar r
    \leq \frac{(\alpha_3\lip_f\alpha_{7}(\dt))^2}{2(\alpha_4-\alpha_3\alpha_{5}(\dt)\lip_{f})}|e(\tau_i)|^2.
\end{align}
where $\tau_1\leq\tau_i<t^*\leq \tau_{i+1}$. By \eqref{eq:Vebis},
\begin{align}\label{eq:Ve2}
|e(\tau_i)|\leq 
\e^{-\frac{\theta \omega\dt }{2}} (c \theta^{q})^2 \lip_{\HH_q} \lip_{\sat}\lip_{\phi} (|z_0| + \lip_{\HH_q}(|x_0| + \|v^*\|_\infty |\mu_0|)).
\end{align}
Set $\theta_3^*\geq\theta_2^*$ to be such that for all $\theta\geq\theta_3^*$ we have
$$
\e^{\theta \omega\dt}>\frac{\alpha_1}{\bar r}
    \left(\frac{\alpha_3\lip_f\alpha_{7}(\dt)}{\alpha_4-\alpha_3\alpha_{5}(\dt)\lip_{f}} (c \theta^{q})^2 \lip_{\HH_q} \lip_{\sat}\lip_{\phi} (|z_0| + \lip_{\HH_q}(|x_0| + \|v^*\|_\infty \bar\lambda))\right)^2.
$$
Then having $\theta\geq\theta_3^*$ contradicts \eqref{eq:Ve1}-\eqref{eq:Ve2}. 
We now assume $\theta\geq\theta_3^*$ and thus $t^*=T_e$, i.e. $x$ remains in $\KK_x'$.

From \eqref{eq:V30}, we get by Grönwall's inequality that for all $i\in\N$ and almost all $t\in(\tau_i,\tau_{i+1})$ such that $t<t^*$,
\begin{align}
    V(x(t))
    &\leq \e^{-(t-\tau_0)\frac{\alpha_4-\alpha_3\alpha_{5}(\dt)\lip_{f}}{2\alpha_1}}
    V(x(\tau_0)) 
    \nonumber\\
    &\quad+\frac{(\alpha_3\lip_f\alpha_{7}(\dt))^2}{2(\alpha_4-\alpha_3\alpha_{5}(\dt)\lip_{f})}
    \sum_{j=0}^i\int_{\tau_j}^{\min(\tau_{j+1}, t)}\e^{-(t-\tau)\frac{\alpha_4-\alpha_3\alpha_{5}(\dt)\lip_{f}}{2\alpha_1}}|e(\tau_j)|^2
    \dd\tau
    \nonumber\\
    &\leq \e^{-(t-\dt)\frac{\alpha_4-\alpha_3\alpha_{5}(\dt)\lip_{f}}{2\alpha_1}}
    V(x(\tau_0)) 
    \nonumber\\
    &\quad+\frac{(\alpha_3\lip_f\alpha_{7}(\dt))^2}{2(\alpha_4-\alpha_3\alpha_{5}(\dt)\lip_{f})}
    \dt
    \Big(
    |e(\tau_i)|^2
    +
    \sum_{j=0}^{i-1}\e^{-(t-\tau_{j+1})\frac{\alpha_4-\alpha_3\alpha_{5}(\dt)\lip_{f}}{2\alpha_1}}|e(\tau_j)|^2
    \Big),
    \label{eq:V6bis}
\end{align}
since $\tau_0<\dt$.

On the other hand, if $\tau_0>0$, then for all $t\in(0, \min(\tau_0, t^*))$, inequality \eqref{eq:V7bis} still holds.
Combining \eqref{eq:V7bis} and \eqref{eq:V6bis}, we obtain that for all $t\in[0, t^*)$,
\begin{align}
    V(x(t))
    &\leq \e^{-(t-\dt)\frac{\alpha_4-\alpha_3\alpha_{5}(\dt)\lip_{f}}{2\alpha_1}}
    \Big(\sqrt{V(x_0)}\e^{\dt\frac{\alpha_3\lip_f}{2\alpha_1}}
    +\dt \frac{\lip_f \|v^*\|_{\infty}}{2\sqrt{\alpha_1}}
    \e^{\dt\frac{\alpha_3\lip_f}{2\alpha_1}}|\mu_0|
    \Big)^2
    \nonumber\\
    &\quad+
    \frac{(\alpha_3\lip_f\alpha_{7}(\dt))^2}{2(\alpha_4-\alpha_3\alpha_{5}(\dt)\lip_{f})}
    \dt
    \Big(
    |e(\tau_i)|^2
    +
    \sum_{j=0}^{i-1}\e^{-(t-\tau_{j+1})\frac{\alpha_4-\alpha_3\alpha_{5}(\dt)\lip_{f}}{2\alpha_1}}|e(\tau_j)|^2
    \Big)
    \label{eq:V20}
\end{align}
By \eqref{eq:Ve}, we get
\begin{align}
    |e(\tau_i)|^2
    &+
    \sum_{j=0}^{i-1}\e^{-(t-\tau_{j+1})\frac{\alpha_4-\alpha_3\alpha_{5}(\dt)\lip_{f}}{2\alpha_1}}|e(\tau_j)|^2
    \nonumber\\
    &\leq
    (
    S+
    \e^{-\theta \omega t}\e^{\theta \omega \dt}
    )
     \e^{\theta \omega \dt} ((c \theta^{q})^2 \lip_{\HH_q} \lip_{\sat}\lip_{\phi})^2 |e(0)|^2.
    \label{eq:V21a}
\end{align}
where $S:=\sum_{j=0}^{i-1}\e^{-(t-\tau_{j+1})\frac{\alpha_4-\alpha_3\alpha_{5}(\dt)\lip_{f}}{2\alpha_1}}\e^{-\theta \omega \tau_j}$.
Moreover, since $\tau_j = \dt-s_0+j\dt$,
\begin{align}
    S
    &= 
    \sum_{j=0}^{i-1}\e^{-(t-(\dt-s_0+(j+1)\dt))\frac{\alpha_4-\alpha_3\alpha_{5}(\dt)\lip_{f}}{2\alpha_1}}\e^{-\theta \omega (\dt-s_0+j\dt)}
    \nonumber\\
    &=
    \e^{-(t-(2\dt-s_0))\frac{\alpha_4-\alpha_3\alpha_{5}(\dt)\lip_{f}}{2\alpha_1}-\theta\omega(\dt-s_0)}
    \sum_{j=0}^{i-1}\e^{j\dt(\frac{\alpha_4-\alpha_3\alpha_{5}(\dt)\lip_{f}}{2\alpha_1}-\theta \omega)}
    \nonumber
\end{align}
Define $\alpha_{9}(\dt, \theta) = \frac{\alpha_4-\alpha_3\alpha_{5}(\dt)\lip_{f}}{2\alpha_1}-\theta \omega$. We distinguish three cases.
\begin{itemize}
    \item If $\alpha_{9}(\dt, \theta)>0$, then
\begin{align}
    S
    &=
    \e^{-(t-(2\dt-s_0))\frac{\alpha_4-\alpha_3\alpha_{5}(\dt)\lip_{f}}{2\alpha_1}-\theta\omega(\dt-s_0)}
    \frac{\e^{i\dt\alpha_{9}(\dt, \theta)}-1}{\e^{\dt\alpha_{9}(\dt, \theta)}-1}
    \nonumber\\
    &\leq
    \e^{(i\dt-t)\frac{\alpha_4-\alpha_3\alpha_{5}(\dt)\lip_{f}}{2\alpha_1}-i\dt\theta \omega}
    \frac{\e^{2\dt\frac{\alpha_4-\alpha_3\alpha_{5}(\dt)\lip_{f}}{2\alpha_1}}
    }{\e^{\dt\alpha_{9}(\dt, \theta)}-1}
    \nonumber\\
    &\leq
    \e^{-i\dt\theta \omega}
    \frac{\e^{2\dt\frac{\alpha_4-\alpha_3\alpha_{5}(\dt)\lip_{f}}{2\alpha_1}}
    }{\e^{\dt\alpha_{9}(\dt, \theta)}-1}
    \nonumber\\
    &\leq
    \e^{-\theta \omega t}
    \frac{\e^{2\dt(\frac{\alpha_4-\alpha_3\alpha_{5}(\dt)\lip_{f}}{2\alpha_1}+\theta\omega)}
    }{\e^{\dt\alpha_{9}(\dt, \theta)}-1};
    \nonumber
\end{align}
\item If $\alpha_{9}(\dt, \theta)<0$, then
\begin{align}
    S
    &=
    \e^{-(t-(2\dt-s_0))\frac{\alpha_4-\alpha_3\alpha_{5}(\dt)\lip_{f}}{2\alpha_1}-\theta\omega(\dt-s_0)}
    \frac{1-\e^{i\dt\alpha_{9}(\dt, \theta)}}{1-\e^{\dt\alpha_{9}(\dt, \theta)}}
    \nonumber\\
    &\leq
    \e^{-t\frac{\alpha_4-\alpha_3\alpha_{5}(\dt)\lip_{f}}{2\alpha_1}}
    \frac{\e^{2\dt\frac{\alpha_4-\alpha_3\alpha_{5}(\dt)\lip_{f}}{2\alpha_1}}
    }{1-\e^{\dt\alpha_{9}(\dt, \theta)}};
    \nonumber
\end{align}
\item If $\alpha_{9}(\dt, \theta)=0$, then
\begin{align}
    S
    &=
    i\e^{-(t-(2\dt-s_0))\frac{\alpha_4-\alpha_3\alpha_{5}(\dt)\lip_{f}}{2\alpha_1}-\theta\omega(\dt-s_0)}
    \nonumber\\
    &\leq
    \frac{t}{\dt}
    \e^{-t\frac{\alpha_4-\alpha_3\alpha_{5}(\dt)\lip_{f}}{2\alpha_1}}
    \e^{2\dt\frac{\alpha_4-\alpha_3\alpha_{5}(\dt)\lip_{f}}{2\alpha_1}}.
    \nonumber
\end{align}
\end{itemize}
Hence, in any case, we get from \eqref{eq:V20}, \eqref{eq:V21a} and the above inequalities that there exist two positive constant $\alpha_{10}(\dt, \theta)$ and $\alpha_{11}(\dt, \theta)$ such that for all $t\in[0, t^*)$,
\begin{align}
    V(x(t))
    &\leq \e^{-t\frac{\alpha_4-\alpha_3\alpha_{5}(\dt)\lip_{f}}{2\alpha_1}}
    \e^{\dt\frac{\alpha_4-\alpha_3\alpha_{5}(\dt)\lip_{f}}{2\alpha_1}}
    \Big(\sqrt{V(x_0)}\e^{\dt\frac{\alpha_3\lip_f}{2\alpha_1}}
    +\dt \frac{\lip_f \|v^*\|_{\infty}}{2\sqrt{\alpha_1}}
    \e^{\dt\frac{\alpha_3\lip_f}{2\alpha_1}}|\mu_0|
    \Big)^2
    \nonumber\\
    &\quad+
    \e^{-t\alpha_{10}(\dt, \theta)}
    \alpha_{11}(\dt, \theta)|e(0)|^2,
    \nonumber
\end{align}
hence,
\begin{align}
    |x(t)|
    &\leq \e^{-t\frac{\alpha_4-\alpha_3\alpha_{5}(\dt)\lip_{f}}{4\alpha_1}}
    \e^{\dt\frac{\alpha_4-\alpha_3\alpha_{5}(\dt)\lip_{f}}{4\alpha_1}}
    \Big(\sqrt{\frac{\alpha_2}{\alpha_1}}\e^{\dt\frac{\alpha_3\lip_f}{2\alpha_1}}
    |x_0|
    +\dt \frac{\lip_f \|v^*\|_{\infty}}{2\alpha_1}
    \e^{\dt\frac{\alpha_3\lip_f}{2\alpha_1}}|\mu_0|
    \Big)
    \nonumber\\
    &\quad+
    \e^{-t\frac{\alpha_{10}(\dt, \theta)}{2}}
    \sqrt{\frac{\alpha_{11}(\dt, \theta)}{\alpha_1}}(
    |z_0| + \lip_{\HH_q}(|x_0| + \|v^*\|_\infty |\mu_0|)).
    \nonumber
\end{align}
Thus, it has been proved that $x$ remains in $\KK_x'$ over $[0, T_e)$, 
and \eqref{eq:exp_x} is satisfied for a suitable choice of $(\bar M, \bar \nu)$.
\end{proof}

With Lemmas \ref{lem:exp_e} and \ref{lem:exp_x}, we are now able to conclude the proof of Theorem \ref{th:main}.

\begin{proof}[Proof of Theorem \ref{th:main}]
We asume that $\Delta\in(0, \Delta_4^*]$ and $\theta>\theta_3^*$ (see Lemma~\ref{lem:exp_x}).
Clearly, $\mu$ remains in $[0, \bar\lambda]$. Moreover,
for all $i\in\N$ such that $\tau_i<T_e$,
\begin{align*}
    \mu^+(\tau_i)
    &= |\lambda(\sat(\phi(z(\tau_i), \Delta, \mu(\tau_i), \RR(\tau_i))))|
    \\
    &\leq |\lambda(\sat(\phi(z(\tau_i), \Delta, \mu(\tau_i), \RR(\tau_i)))) - \lambda(x(t))| + |\lambda(x(t))|
    \\
    &= |\lambda(\sat(\phi(z(\tau_i), \Delta, \mu(\tau_i), \RR(\tau_i)))) - \lambda(\sat(\phi(\HH_q(x(\tau_i), \dt, \mu(\tau_i)\RR(\tau_i) v^*), \dt, \mu(\tau_i), \RR(\tau_i))))|
    \\
    &\quad+ |\lambda(x(t))|
    \\
    &\leq \lip_\lambda\lip_{\sat}\lip_\phi|z(\tau_i)-\HH_q(x(\tau_i), \dt, \mu(\tau_i)\RR(\tau_i) v^*)| + \lip_\lambda|x(\tau_i)|
    \\
    &= \lip_\lambda\lip_{\sat}\lip_\phi|e(\tau_i)| + \lip_\lambda|x(\tau_i)|
\end{align*}
Since $e$ and $x$ are exponentially stable at $0$ by Lemmas \ref{lem:exp_e} and \ref{lem:exp_x}, so is $\mu$.
Finally, for all $(t,i)\in E$,
\begin{align*}
    |z(t, i)|
    &\leq |e(t, i)| + |\HH_q(x(t, i), s(t, i), \mu(t, i)\RR(t, i) v^*)|
    \\
    &\leq |e(t, i)| + |\HH_q(x(t, i), \dt, \mu(t, i)\RR(t, i) v^*) - \HH_q(0, \dt, 0)|
    \\
    &\leq |e(t, i)| + \lip_{\HH_q}(|x(t, i)| + \|v^*\|_\infty |\mu(t, i)|)
\end{align*}
Since $\mu$, $e$ and $x$ are bounded over $[0, T_e)$ by the reasoning above and Lemmas \ref{lem:exp_e} and \ref{lem:exp_x} respectively, so is $z$.
Finally, $(x,z,s,\mu,\RR)$ remains bounded, thus $T_e=+\infty$.
Hence, according to the steps above, $e$, $x$ and $\mu$ are
exponentially stable at $0$.
Hence, $z$ is also exponentially stable at $0$, which concludes the proof of Theorem \ref{th:main}.
\end{proof}

\section{Universality theorem}
\label{sec:thuniv}

Universal inputs are defined by Sussmann in \cite{MR0541865} as inputs for which the resulting output of a system such as \eqref{syst} allows to distinguish between any two states that may be distinguished. In other words, if $v$ is a universal input, for two points $x_a, x_b\in \R^n$, if there exists an input $u$ for which $t\mapsto h(X(x_a,t,u))$ and $t\mapsto h(X(x_a,t,u))$ do not coincide, then $t\mapsto h(X(x_a,t,v))$ and $t\mapsto h(X(x_a,t,v))$ should also not coincide.
In \cite{MR0541865}, it is proved that for analytic systems
there always exist analytic inputs, and that universality is a generic property.
In this section, we lift arguments from \cite{MR0541865} to provide a proof of our own universality theorem, Theorem~\ref{lem:sus_bis}. It is a stronger result that follows from stronger assumptions (mainly that the system is strongly differentialy observable for the null-input $0$).
In \cite{MR0541865}, Sussmann works under numbered assumptions (H1) to (H4). Assumptions (H1)-(H3)-(H4) correspond in our case to our analyticity assumption. For our needs, we switch (H2) for the more restrictive assumption that the set of admissible values for the controls is $\R^p$.

We assume in the rest of the section that $\T>0$ is a fixed time frame for which we discuss the properties of inputs in the set $C^\infty([0,\T],\R^p)$ endowed with the compact-open topology.

\begin{definition}
Let $k$ be a positive integer. We say that a jet $\sigma\in \left(\R^p\right)^{k} $ \emph{distinguishes} a pair  $(x_a,x_b)$ in $\R^n\times \R^n$,  if there exists an integer $0\leq j\leq k$ such that  $H_{j}(x_a,\sigma)\neq H_{j}(x_b,\sigma)$. By extension, an input $u\in C^\infty([0,\T],\R^p)$ is said to \emph{distinguish} $(x_a,x_b)$, if there exists a positive integer $k$, a time $t\in [0,\T]$ such that $\mathcal{H}_k(x_a,t,u)\neq \mathcal{H}_k(x_b,t,u)$. That is, the jet of order $k-1$ of $u$ at $t$ distinguishes $(x_a,x_b)$. A pair $(x_a,x_b)\in \R^n\times \R^n$ is said to be \emph{distinguishable} if there exists a jet (that can be of any order) that distinguishes it.
\end{definition}

Let $\diag=\{(x,x)\mid x\in \R^n\}$ denote the diagonal of the state space. Any pair $(x,x)$ cannot be distinguished, but
under Assumption~\ref{ass:obs0}, any pair in $(\R^n\times \R^n)\setminus \diag$ is distinguishable.  We also define 
for any $\eta>0$, a neighborhood of the diagonal
\begin{equation}\label{E:defDeta}
\diag_\eta=\{(x_a,x_b)\in {(\R^n)}^2\mid |x_a-x_b|\leq \eta\}.    
\end{equation}
Over the whole section, $\Kp$ denotes pairs of points of the state space $\R^n$. For sets $\Kp\subset {(\R^n)}^2$ we define the notation $\uKp$ to denote the projection of $\Kp$ onto $\R^n$, with respect to both the first and second element in the pair, that is:
\begin{equation}\label{E:defuKp}
    \uKp=
    \left\{
        x\in \R^n\mid \left(\{x\}\times\R^n
        \cup \R^n\times \{x\}\right)\cap \Kp\neq \emptyset
    \right\}.
\end{equation}
Naturally, if $\Kp$ is compact, $\uKp$ is also compact.
\begin{definition}\label{D:Z'}
The input $u\in C^\infty([0,\T],\R^p)$ is said to be universal over $\Kp\subset \R^n\times \R^n$ if it distinguishes any distinguishable pair in $\Kp$. 
Following Sussmann,  we denote by $Z(\T,\Kp)$ 
the set of inputs that are universal 
over $\Kp$.

For our purposes, we denote by $Z'(\T,\Kp)$ the set of inputs $v\in C^\infty([0,\T],\R^p)$ such that there exist
$T\in (0,\T]$, $q$ integer, for which the map
$(x_a,x_b)\mapsto \mathcal{H}_{q}(x_b,t,\mu \RR v)-\mathcal{H}_{q}(x_a,t,\mu \RR v)$
never vanishes over $\Kp\setminus \diag$, for all $t\in [0,T]$, for all $\mu \in [0,1]$ and all $\RR\in \O(p)$, and $\frac{\partial\mathcal{H}_{q}}{\partial x}(x,t,\mu \RR v)$ is injective for all $x\in \uKp$, all $t\in [0,T]$, all $\mu \in [0,1]$ and all $\RR\in \O(p)$

\end{definition}

\begin{remark}\label{R:defVueps}
For any $u\in C^\infty([0,\T],\R^p)$, $k\in \N$, $\varepsilon_0,\dots,\varepsilon_k$ positive constants, we define
\begin{equation}\label{E:defVueps}
    V(u,\varepsilon_0,\dots,\varepsilon_k)
    =
    \left\{
    v\mid
    \|v^{(j)}-u^{(j)}\|_{\infty}<\varepsilon_j,\text{~for all~} j=0, \dots, k
    \right\}.
\end{equation}
The family of these sets is a fundamental system of neighborhoods of $u$ in the compact-open topology.
\end{remark}

\begin{lemma}
\label{L:ZTKopen}
If $\Kp\subset \R^n\times \R^n$ is compact then $Z'(\T,\Kp)$ is open for the compact-open topology.
\end{lemma}
\begin{proof}
Fix any $v\in Z'(\T,\Kp)$, $T>0$ and $q$ integer as in Definition~\ref{D:Z'}.
By construction (see Section~\ref{sec:def-Hq}), $\mathcal{H}_q$ is an analytic function of the jets of order $q-1$ of $v$, and therefore is continuous with respect to $(x,t,\mu,\RR,u)\subset \R^n\times [0,T]\times [0,1]\times\O(p)\times C^\infty([0,\T],\R^p)$, with each finite dimensional set in the product assumed to be endowed with their canonical topology, and the compact-open topology for $C^\infty([0,\T],\R^p)$. The same holds for $\frac{\partial \mathcal{H}_q}{\partial x}$.
Applying Lemma \ref{lem:iminj}, items (\ref{item1})-(\ref{item2}), with $\Pi=[0,T]\times [0,1]\times\O(p)\times C^\infty([0,\T],\R^p)$ and $F_\pi(x)=\HH_{q}(x,t,\mu\RR u)$ for $\pi=(t,\mu,\RR, u)$ yields the existence of 
$$
\rho(u)=\min_{\pi\in [0,T]\times [0,1]\times\O(p)\times \{u\}} \min\big(\rho_2(\pi),\rho_1(\pi,\eta_2(\pi)/2)\big) 
$$
continuous with respect to $u$ and such that 
\[
\left|
    \HH_{q}(x_a,t,\mu\RR u)
    -\HH_{q}(x_b,t,\mu\RR u)
\right|
\geq \rho(u)|x_a-x_b|.
\]
Since $v\in Z'(\T,\Kp)$, $\rho(v)>0$ and there exists a $\varepsilon_0, \dots ,\varepsilon_{q-1}$ such that if $u\in V(v,\varepsilon_0,\dots ,\varepsilon_{q-1})$ (as in \eqref{E:defVueps}), then $\rho(u)>0$.
Consequently,
for any $u\in V(v,\varepsilon_0,\dots ,\varepsilon_{q-1})$, $t\in [0,T]$, $\mu\in[0,1]$, $\RR\in \O(p)$, $(x_a,x_b)\mapsto \mathcal{H}_{q}(x_b,t,\mu \RR v)-\mathcal{H}_{q}(x_a,t,\mu \RR u)$ never vanishes over  $\Kp\setminus\diag$. Up to reducing $\varepsilon_0,\dots ,\varepsilon_{q-1}$, we also recover the injectivity of the Jacobian over $\uKp$,
which concludes the openness argument.
\end{proof}

\begin{remark}
In \cite{MR0541865}, openness of the set $Z(\T,\Kp)$ of universal inputs over $\Kp$ cannot be proved because of the issue at the diagonal. Even if, for a given $v$, $(x_a,x_b)\mapsto \mathcal{H}_{q}(x_b,0, v)-\mathcal{H}_{q}(x_a,0,v)$ vanishes only when $x_a=x_b$ in $\Kp$, there may exists arbitrarily small perturbations $\delta v$ such that $(x_a,x_b)\mapsto\mathcal{H}_{q}(x_b,0, v+\delta v)-\mathcal{H}_{q}(x_a,0,v+\delta v)$ has nontrivial vanishing points off the diagonal (as with the Whitney pleat for instance \cite{MR2896292}). The injectivity requirement on the Jacobian  of $\mathcal{H}_{q}$ is crucial to avoid that difficulty.
\end{remark}

In \cite{MR0541865}, the following is proved (see Theorem~2.2 and proof of Theorem~2.2 in Section~3).

\begin{theorem}[Sussman's generic universality theorem, restated]\label{thm:Sus_restated}
    Under analyticity assumption.
    Let $\T>0$. Let $\Kp$ be a compact in the set of distinguishable pairs in $\R^n\times \R^n$. The set $\mathring{Z}(\T,\Kp)$ is dense in $C^\infty([0,\T],\R^p)$. As a consequence, in $C^\infty([0,\T],\R^p)$, the set of inputs that distinguish between any distinguishable pair $\R^n\times \R^n$,  contains a  countable intersection of open sets that is dense in $C^\infty([0,\T],\R^p)$.
\end{theorem}

Following a similar reasoning to the proof of this theorem, we are able, {\it mutatis mutandis}, to prove Theorem~\ref{lem:sus_bis}. From the above result, we are looking to add the strong differential observability (immersion property) as well as the introduction of some uniformity in the universality with respect to
the parameters $\mu$ and $\RR$ (lying in a compact set).
Theorem~\ref{thm:Sus_restated} follows from a sequence of intermediary lemmas and and theorems, namely
\cite[Lemma~3.1 to Theorem~3.9]{MR0541865}.
While the 
parameters side 
can be achieved by a refinement on the proof of
\cite[Theorem~3.9]{MR0541865},
the immersion property requires a finer approach. For this, we follow the proof given in  \cite{MR0541865} and propose changes where necessary. We try to be as complete as possible, however, for shortness of presentation, some elements will remain black boxes. The reader is directed towards~\cite{MR0541865} for more details.

Let us introduce some notations of Sussmann's proof. 
Set $\Gamma_0=\R^n\times \R^n$ and $\Gamma_k=\R^n\times \R^n\times (\R^p)^{k}$.
Define $G_k:\Gamma_k\to \R^p$ by $G_k(x_a,x_b,\sigma)=H_k(x_a,\sigma)-H_k(x_b,\sigma)$.
If $j\leq k$, we set $\mu_{kj}:\Gamma_k\to \Gamma_j$ to be the trivial projection. For any nonnegative integer $k$, consider the set
$$
A_k=\left\{(x_a,x_b,\sigma)\in \Gamma_k \mid G_j\circ \mu_{kj}(x_a,x_b,\sigma)=0, \forall j\leq k\right\}.
$$
The set $A_k$ denotes the set of jets of bad inputs, observability-wise, as seen in the following statement (from \cite{MR0541865}). 
\begin{proposition}[Sussmann's bad jets lemma]
\label{P:Sussmann_jets-to-universal-inputs}
Let $k$ be a non-negative integer.
  \begin{enumerate}
      \item $A_k$ is an analytic subset of $\Gamma_k$.
      \item If $j\leq k$, $\mu_{kj}(A_k)\subset A_j$.
       \item $(x_a,x_b,\sigma)\in A_k$ if and only if for any $C^\infty$ input $u$ with $(k-1)$-jet $\sigma$ at $0$,
      $$
      \left.\frac{\dd^{j}}{\dd t^j}\right|_{t=0}h(X(x_a,t,u))
      =
      \left.\frac{\dd^{j}}{\dd t^j}\right|_{t=0}h(X(x_b,t,u)),
      \qquad 
      0\leq j\leq k.
      $$
  \end{enumerate}  
\end{proposition}

To include an immersion condition, we propose to define another family of sets that maintain %
most of the properties of the family $(A_k)$.
We define 
$$
A'_k=\left\{
(x_a,x_b,\sigma)\in \Gamma_k
~\left|~
\begin{aligned}
    &G_j\circ \mu_{kj}(x_a,x_b,\sigma)=0, \forall j\leq k \text{, or}
    \\
    &\rank\left(
    \frac{\partial}{\partial x_i} G_1(x_a,x_b,\sigma),\dots,\frac{\partial}{\partial x_i} G_k(x_a,x_b,\sigma)
    \right)<n, \text{ for } i=a \text{ or }b
\end{aligned}
\right.
\right\}.
$$
The immediate counterpart of Proposition~\ref{P:Sussmann_jets-to-universal-inputs} is as follows.

\begin{proposition}[Our bad jets lemma]
\label{P:Sussmann_jets-to-universal-inputs_2}
Let $k$ be a non-negative integer.
  \begin{enumerate}
      \item $A'_k$ is an analytic subset of $\Gamma_k$.
      \item If $j\leq k$, $\mu_{kj}(A'_k)\subset A'_j$.
      \item $(x_a,x_b,\sigma)\in A_k$ if and only if for any $C^\infty$ input $u$ with $(k-1)$-jet $\sigma$, either there exists $j$, $0\leq j\leq k$, such that
      $$
      \left.\frac{\dd^{j}}{\dd t^j}\right|_{t=0}h(X(x_a,t,u))
      =
      \left.\frac{\dd^{j}}{\dd t^j}\right|_{t=0}h(X(x_b,t,u))
      \qquad 
      0\leq j\leq k
      $$
      or, for either index $i=a$ or $i=b$,
      $$
      \rank\left(\left.\frac{\dd^{j}}{\dd t^j}\right|_{t=0}
      \frac{\partial h}{\partial x}(X(x_i,t,u))\right)_{0\leq j\leq k}<n.
      $$
  \end{enumerate}  
\end{proposition}

\begin{remark}\label{rem:rank}
Recall the important rule that is used in subsequent proofs: a family $(W_1,\dots W_k)\in(\R^n)^k$ satisfies $\rank(W_1,\dots, W_k)<n$ if and only if there exists $\xi\in \S^{n-1}$ such that $\langle W_j,\xi\rangle=0$ for all $1\leq j\leq k$. Likewise, $\rank(W_1,\dots, W_k)=n$ if and only if for any $\xi\in \S^{n-1}$, there exists $1\leq j\leq k$ such that $\langle W_j,\xi\rangle\neq 0$.
\end{remark}

Let $\mathcal{E}$ denote a family of subsets $E_k$ of $\Gamma_k$ such that
\begin{enumerate}
    \item $E_k$ is an open subset of $\Gamma_k$
    \item $E_0$ is contained in  the set of distinguishable pairs of points. (When Assumption~\ref{ass:obs0} holds, all pairs are assumed to be distinguishable, so this condition is essentially empty.)
    \item $\mu_{jk}(E_j)\subset E_k$ when $j\geq k$.
    \item $A_k\cap E_k$ is a union of a finite number of connected analytic submanifolds of $\Gamma_k$.
    \item[$4'.$] $A'_k\cap E_k$ is a union of a finite number of connected analytic submanifolds of $\Gamma_k$.
\end{enumerate}

The technical heart of the proof is probably the following result, showing that the amount of space that the sets $A_k$ occupy within $E_k$ becomes smaller as $k$ grows. (Meaning that it is more difficult for a jet of high order to not distinguish between two points). In a later proof, this will allow to invoke  Sard's lemma type arguments to show that  jets having sufficiently high order should be universal (at least over a compact).
Let us recall the notions of dimension and codimension used in \cite{MR0541865}.
If $A$ is a subset of a smooth manifold $E$, its dimension is the maximal dimension of any smooth submanifold of $E$ included in $A$, and is denoted by $\dim A$. The codimension of $A$ is $\codim_E A = \dim E - \dim A$.

\begin{proposition}[Sussmann's growth lemma]
\label{P:growth}
Let $c_{\mathcal{E}}(k)=\mathrm{codim}_{E_k}(A_k\cap E_k)$. Assume that items 1 though 4 hold for $\mathcal{E}$. Then $\lim_{k\to \infty}c_{\mathcal{E}}(k)=+\infty$.
\end{proposition}

This result does not hold, as is, for the family $(A'_k)$, but we can provide new arguments to extend the statement to include it. More precisely, we need an equivalent of the following property.

\begin{lemma}[{\cite[Lemma 3.5]{MR0541865}}]
\label{L:Sus_extension}
Let $(x_a,x_b)$ be a distinguishable pair,
let $k$ be a nonnegative integer, and let $\sigma\in  (\R^p)^{k+1}$ be a $k$-jet. Then there exists $j\geq k$ and a $j$-jet $\tau\in (\R^p)^{j+1}$ such that $\sigma$ is the projection of the $j$-jet $\tau$ onto the space of $k$-jets, and $\tau$ distinguishes $(x_a, x_b)$.
\end{lemma}

An equivalent formulation is that if $(x_a,x_b,\sigma)\in A_k$, there always exist $j>k$, $\tau$ such that $(x_a,x_b,\tau)\notin A_j$ and $\mu_{jk}(x_a,x_b,\tau)=(x_a,x_b,\sigma)$. Using the same method of proof as Lemma~\ref{L:Sus_extension}, we are able to achieve the following.

\begin{lemma}[Jet extension lemma]
\label{L:Our_extension}
Let Assumption~\ref{ass:immersion0} hold.
Let $(x_a,x_b)$ be a distinguishable pair,
let $k$ be a nonnegative integer, and let $\sigma\in (\R^p)^{k}$ be a $(k-1)$-jet. Either $(x_a,x_b,\sigma)\notin A'_k$ or there exist $j>k$ and $\tau$ such that $(x_a,x_b,\tau)\notin A'_j$ and $\mu_{jk}(x_a,x_b,\tau)=(x_a,x_b,\sigma)$.
\end{lemma}
\begin{proof}
Let $(x_a,x_b,\sigma)\in \Gamma_k$. If $(x_a,x_b,\sigma)\notin A'_k$, there is nothing to prove. If $(x_a,x_b,\sigma)\in A_k$, we can apply Lemma~\ref{L:Sus_extension}, and find $j>k$, $\tau$ such that $\mu_{jk}(\tau)=\sigma$ and $(x_a,x_b,\tau)\notin A_j$. Hence without loss of generality we assume that $(x_a,x_b,\sigma)\in A'_k\setminus A_k$. Since $\mu_{jk}(A_j)\subset A_k$, we know that any extension $\tau$ of $\sigma$ will satisfy $(x_a,x_b,\tau)\notin A_j$. Hence we can focus on showing that for any $(x_a,x_b,\sigma)\in A'_k\setminus A_k$ there exists an extension $\tau$ of $\sigma$ such that $(x_a,x_b,\tau)\notin A'_j$ by proving that the rank condition is satisfied
(following Remark~\ref{rem:rank}).

Let $u\in C^\infty([0,\T],\R^p)$ be any realization of the jet $\sigma$.
Let $T\in(0, \T)$ be such that the Cauchy problems \eqref{syst}
with initial conditions $x_a$ and $x_b$ admits solutions $X(x_a, \cdot,u)$ and $X(x_b, \cdot,u)$, respectively, over $[0, T]$.
Set $x_a'=X(x_a, T,u)$ and $x_b'=X(x_b, T,u)$.
Let $r$ be a positive constant and set $\bar v = 2|u(T)|$.
For any $v\in C^0\big(\R_+, B_{\R^p}(0, \bar v)\big)$,
set $T_e(v, x_i')=\inf\{t>0\mid |X(t, x_i',v)-x'_i|>r\}$ for $i=a$ and $i=b$.
If $T_e(v, x_i')<\infty$, then,
\[
\begin{aligned}
r
=
|X(T_e(v, x_i'), x_i',v)-x_i'|&\leq \int_0^{T_e(v, x_i')} |f(X(s, x_i', v),v(s))|\dd s
\\
&\leq T_e(v, x_i')
\sup_{B_{\R^n}(x'_i, r) \times B_{\R^p}(0, \bar v)} |f(x,u)|,
\end{aligned}
\]
and thus,
\(
T_e(v,x_i')
\ge
\eta \defeq
r/\max_{i\in\{1, 2\}}\sup_{B_{\R^n}(x_i', r)\times B_{\R^p}(0, \bar v)}|f(x,u)|.
\)
By construction, for any input $v\in C^0\big(\R_+, B_{\R^p}(0, \bar v)\big)$,
any solution to the Cauchy problems \eqref{syst}
with initial conditions  $x_a'$ and $x_b'$ admits solutions $X(x_a', \cdot,v)$ and $X(x_b', \cdot,v)$, respectively, over $[0, \eta]$.
Furthermore, for all $t\in [0,\eta]$,  $i=a,b$, $|X(x_i', t,v)-x_i'|{\le} r$.

Set $T'>0$ to be a positive time such that all solutions to $\dot{x}=f(x,0)$ with initial condition in $B_{\R^n}(x'_1, r)\cup B_{\R^n}(x'_2, r)$ exist over $[0,T']$.
Let now $K$ be a compact set that contains in its interior $X(x_i, t,u)$, $i=a,b$, $t\in[0,T]$, as well as $X(B_{\R^n}(x'_i, r),t,0)$, $i=a,b$, $t\in[0,T'']$.
By Assumption~\ref{ass:immersion0},
since $K$ is compact, there exists $k_0$ such that $\mathcal{H}_{k_0}(\cdot,0,0)$ is an immersion at $x$ for any $x\in K$. Indeed one such $k_0(x)$ exists for any $x\in K$ by assumption.
This property being open, we get a uniform $k_0$ over $K$ by 
a finite open cover argument.

Let $v_\eta\in C^{k_0}([0,\eta],\R^p)$ be such that the $k_0$-jet at 0 is given by the $k_0$-jet at $T$ of $u$, and the $k_0$-jet at $\eta$ is $0$. 
We also require that  $\|v_\eta\|_{\infty}\leq \bar v= 2|u(T)|$.
Now we consider $v\in C^{k_0}([0,T+T'+\eta],\R^p)$, a $k_0$-continuously differentiable extension of the input $u$:
$$
v(t)=
\begin{cases}
    u(t)        & \text{if } t\in [0,T)
    \\
    v_\eta(t-T')  & \text{if } t\in [T,T+\eta)
    \\
    0           & \text{if } t\in [T+\eta,T+T'+\eta]
\end{cases}
$$
By construction, $v$ is admissible,  $X(x_a, t,v)$ and $X(x_b, t,v)$ are well defined  for all $t\in[0, T+T'+\eta]$ and belong to $\mathring{K}$.

We then write $v(t)=P(t)+t^{k_0+1}Q(t)$, where $P(t)$ is the Taylor polynomial of degree $k_0$ of $v$ at $0$. In fact, $P(t)$ is the polynomial realization of the $k_0$-jet $\sigma$. Using Weierstrass approximation theorem of order $k_0$, we can find a sequence of polynomials $Q_\ell$ such that for all $0\leq m\leq k_0$, $Q_\ell^{(m)}$ converges uniformly towards $Q^{(m)}$ as $\ell\to \infty$. As such, $v_\ell=P(t)+t^{k_0+1}Q_\ell(t)$ is a sequence of polynomials with $v_\ell^{(m)}$ converging uniformly towards $v^{(m)}$ ($m\leq k_0$) while $v_\ell$ maintains the $k$-jet $\sigma$ at $0$.

Because being an immersion is an open property,
there exists $\delta>0$ such that if an input $w\in C^{k_0-1}([0,T],\R^p)$ satisfies $\|w^{(m)}\|_{\infty}\leq \delta$  for all $0\leq m\leq k_0-1 $, then $x\mapsto \mathcal{H}_{k_0}(x,t,w)$ is an immersion at all $x\in K$, for all $t\in [0,T]$. Since $v_\ell^{(m)}\to 0$ on $[T'+\eta,T'+T''+\eta]$, we have that for $\ell$ large enough, $x\mapsto \mathcal{H}_{k_0}(x,t,v_\ell)$ is an immersion at all $x\in K$ and all $t\in [T+\eta,T+T'+\eta]$. For $\ell$ large enough we also have that 
$X(x_i,t,v_\ell)\in \mathring{K}$, for all $t\in[0,T+T'+\eta]$, $i=a,b$. This allows to conclude that for all $\xi\in\S^{n-1}$, and any $i=a,b$, there exists $0\leq l\leq k_0$
$$
\lang
\left.\frac{\dd^{l}}{\dd t^l}\right|_{t=T+\eta}
\frac{\partial h}{\partial x}(X(x_i,t,v_\ell)), \xi
\rang
\neq0.
$$
Now let us deduce what happens at $t=0$.
For any $(\xi, i)\in\S^{n-1}\times\{1, 2\}$,
the map
$t\mapsto \lang \frac{\partial h}{\partial x}(X(x_i,t,v_\ell)), \xi \rang$
is analytic (recall $v_\ell$ is actually a polynomial, approximating a non-analytic function) and non zero since it is non-zero on $[T+\eta,T+T'+\eta)$. 
Hence for any $i=a,b$, and any $\xi\in\S^{n-1}$, there exists $l_i(\xi)\geq0$ such that 
$$
\lang
\left.\frac{\dd^{l_i(\xi)}}{\dd t^{l_i(\xi)}}\right|_{t=0}
 \frac{\partial h}{\partial x}(X(x_i,t,v_\ell)),\xi
\rang
\neq 0.
$$
By continuity of the scalar product,
we again get an open property, and are able to use compactness of $\S^{n-1}\times\{1, 2\}$ to conclude that there exists $j\geq 0$ such that $l_i(\xi)\leq j$ for all $(\xi, i) \in \S^{n-1}\times\{1, 2\}$.
In conclusion, for $\ell$ large enough, the $j$-jet $\tau$ of $v_\ell$ satisfies the required condition: $\mu_{jk}(x_a,x_b,\tau)=(x_a,x_b,\sigma)$ and $(x_a,x_b,\tau)\notin A'_j$.
\end{proof}

These elements are sufficient to add the strong differential observability to Sussmann's result, that is, differential observability and the immersion property of the map of derivatives of the output. We recover the same result regarding the codimension of the spaces $(A'_k)$.

\begin{corollary}[Our growth lemma]
\label{C:growth2}
Let $c'_{\mathcal{E}}(k)=\mathrm{codim}_{E_k}(A'_k\cap E_k)$. Assume that items 1 through $4'$ hold for $\mathcal{E}$. Under assumption~\ref{ass:immersion0}, $\lim_{k\to \infty}c'_{\mathcal{E}}(k)=+\infty$.
\end{corollary}

\begin{proof}
The proof of this statement is exactly the same as the proof of Proposition~\ref{P:growth}. We just need to substitute Proposition~\ref{P:Sussmann_jets-to-universal-inputs} (\cite[Lemmas 3.7, 3.6 (ii)]{MR0541865}) with Proposition~\ref{P:Sussmann_jets-to-universal-inputs_2}, and Lemma~\ref{L:Sus_extension} (\cite[Lemma 3.5]{MR0541865}) with Lemma~\ref{L:Our_extension}.
\end{proof}

Now we propose a refinement of Sussmann's proof that includes the contraction coefficient $\mu$ and the action of the special orthogonal group $\O(p)$ defined by:
$\RR\cdot \sigma=(\RR\sigma_1,\dots, \RR\sigma_k)$
for any $(\RR, \sigma)\in \O(p)\times (\R^p)^k$ and any $k\geq 1$.

\begin{proposition}
\label{P:jets}
Let Assumptions \ref{ass:obs0} and \ref{ass:immersion0} hold. 
Let $\Kp$ be a compact subset of $\R^n\times \R^n\setminus \diag$.
Let $(B_\ell)_{\ell\geq 0}$ be a sequence of open balls of positive radii and centered at $0 \in\R^p$.
Then there exists $k>0$ and an open and dense subset $W$ of $B_0 \times B_1\times \cdots\times B_{k-1}$ such that,
whenever $(x_a, x_b, \sigma, \mu, \RR)\in \Kp \times W \times [0, 1] \times \O(p)$,
then $(x_a,x_b,\mu \RR\cdot\sigma)\notin  A'_k$.
\end{proposition}
\begin{proof}
We follow the reasoning of Sussmann's Theorem 3.9.  The first step is to restrict the search by stipulating that the result only needs to hold true for some neighborhood in $\R^n\times \R^n\setminus \diag$ of any pair $(x_a,x_b)$  such that $x_a\neq x_b$. Then the statement holds on any compact subset $\Kp$.

On the one hand, by Assumptions~\ref{ass:obs0} and \ref{ass:immersion0},
there exists $k'\geq 0$ such that,  $(x_a,x_b,0,\dots,0)\notin A_{k}'$ for all $k\ge k'$.
Because $A_{k}'$ is closed, it follows that for any $k> k'$, there exists $\varepsilon_k>0$ such that,
$A_k'\cap
\big[E_0\times B_{\R^{pk}}(0,\varepsilon_k) \big]
=\emptyset$ and
$\overline{E}_0
\subset \R^n\times \R^n\setminus \diag$
taking $E_0= B_{\R^{n}}(x_a,\varepsilon_k)\times B_{\R^{n}}(x_b,\varepsilon_k)$.

For $k\geq 1$, set $E_k=E_0\times B_0\times \cdots\times B_{k-1} $.
Then the induced family $\mathcal{E}= (E_\ell)_{\ell\ge 0}$ satisfies conditions 1 to 3 of Proposition~\ref{P:growth} trivially. Condition 4-4' also hold since the set $E_0,B_0$ were picked analytic.

On the other hand,
by Corollary~\ref{C:growth2}, there exists $k''\ge k'$ such that
\(
c'_{\mathcal{E}}(k)>2n +1 +p(p-1)/2
\)
for all $k\ge k''$.
Fix $k> k''$ and let
\(
\varsigma: \Gamma_k \times [0, 1] \times \O(p) \to {\R^{pk}}
\)
be the surjective analytic mapping given by
\( \varsigma(x_a, x_b, \sigma, \mu, \RR)=\mu\RR\cdot\sigma \).
The compact set $L$,
image under $\varsigma$ of $\overline{A_k'\cap E_k} \times [0, 1] \times \O(p)$ is a semianalytic set of (strictly) positive codimension.
By Sard's Lemma, $L$ has zero measure, implying that $\R^{pk}\setminus L$ is dense.
Put
\[
W=
\left[
    B_{\R^{pk}}(0,\varepsilon_k) \cup
    \R^{kp} \setminus L
\right]
\cap 
\left[
    B_0\times \cdots\times B_{k-1}
\right].
\]
Then $W$ is clearly open and dense in $B_0\times \cdots\times B_{k-1}$.
Now let $(x_a', x_b', \mu, \RR)\in E_0 \times [0,1]\times \O(p)$
and assume that $ \sigma \in W$.
If $\sigma\in  B_{\R^{pk}}(0,\varepsilon_k)$,
then $\mu \RR\cdot\sigma\in  B_{\R^{pk}}(0,\varepsilon_k)$ and $\mu \RR\cdot \sigma\notin L$, so that $(x_a',x_b',\mu \RR\cdot\sigma)\notin A'_k\cap E_k$.
If $\sigma\neq  B_{\R^{pk}}(0,\varepsilon_k)$,
we can assume that 
$\mu\neq 0$ since $(x_a',x_b',0)\notin A'_k$.
Clearly, $\mu\RR\cdot\sigma\in B_0\times\cdots\times B_{k-1}$,
Let
$\nu:\R^{pk}\setminus\{0\}\to \S^{pk-1}$ be given by $\nu(x)=x/|x|$.
Because $0 \in L$,
the image of $\R^{pk} \setminus L$ under $\nu$ is well-defined.
Suppose for the sake of contradiction that
$(x_a', x_b', \mu\RR\cdot\sigma)\in A'_k\cap E_k$.
This implies that $\mu\RR\cdot\sigma \in L\setminus\{0\}$
and thus
$\nu(\sigma) \in \nu(\mu^{-1}\RR^{-1}\cdot L\setminus\{0\})=\nu(L\setminus\{0\})$.
As a result,
$\nu(\sigma) \in  \nu((\R^p)^k \setminus L) \cap \nu(L\setminus\{0\})=\emptyset$,
thus $\sigma=0$, which is a contradiction.
\end{proof}

Like Sussmann's \cite{MR0541865}, the last step of the proof is to move from spaces of jets to spaces of functions. Recall that $Z'(\T,\Kp)$ denotes the set of inputs $v\in C^\infty([0,\T],\R^p)$ such that there exist
$T\in (0,\T]$, $q$ integer, for which the map $(x_a,x_b)\mapsto \mathcal{H}_{q}(x_a,t,\mu \RR v)-\mathcal{H}_{q}(x_b,t,\mu \RR v)$ is an immersion that never vanishes over $\Kp$, for all $t\in [0,T]$, for all $\mu \in [0,1]$ and all $\RR\in \O(p)$. Let us first show that we can safely move from a good jet (one that avoids a critical set $A_k'$) to a good input (one that belongs to a $Z'(\T,\Kp)$).

\begin{lemma}\label{L:good_jet_to_good_input}
Let $\Kp$ be a compact subset of $\R^n\times \R^n\setminus \diag$. Let $k\in \N$ and let $\sigma\in (\R^p)^k$ be a $(k-1)$-jet such that for all $(x_a,x_b,\mu,\RR)\in \Kp\times [0,1]\times \O(p)$, $(x_a,x_b,\mu \RR\cdot \sigma)\notin A'_k$. Then any $u\in C^\infty ([0,T],\R^p)$ having its $(k-1)$-jet at $0$ given by $\sigma$ belongs to $Z'(\T,\Kp)$.
\end{lemma}
\begin{proof}
Let $\sigma(t)$ denote the $(k-1)$-jet of $u$ at time $t$. It describes a smooth trajectory in the space $(\R^p)^k$ such that $\sigma(0)=\sigma$. Assume by contradiction that $u\notin Z'(\T,\Kp)$. By point 3 of Proposition~\ref{P:Sussmann_jets-to-universal-inputs_2}, it follows that for any $N\in \N$, there exists $t^N\in (0,\T/N]$, $(x_a^N,x_b^N,\mu^N,\RR^N)\in \Kp\times [0,1]\times \O(p)$ such that  $(x_a^N,x_b^N,\mu^N\RR^N\cdot \sigma(t^N))\in A'_k$.
The sequence $(x_a^N,x_b^N,\mu^N,\RR^N)$ converges up to extraction as $N\to \infty$, since it belongs to a compact set. We denote the limit with $(x_a^\infty,x_b^\infty,\mu^\infty,\RR^\infty)\in \Kp\times [0,1]\times \O(p)$. Also, $\sigma(t_N)\to \sigma$. 
The set $A'_k$ is closed, hence $(x_a^\infty,x_b^\infty,\mu^\infty\RR^\infty\cdot \sigma)\in A'_k$, which is a contradiction.
\end{proof}

We now combine Proposition~\ref{P:jets} and Lemma~\ref{L:good_jet_to_good_input} to prove the density of good inputs if we avoid the diagonal.

\begin{lemma}\label{L:ZisDense}
Let $\Kp\subset (\R^n\times \R^n)\setminus\diag$ be a compact set. The set $Z'(\T,\Kp)$ is dense in $C^{\infty}([0,T],\R^p)$ in the compact-open topology.
\end{lemma}

\begin{proof}
We use the fundamental system of neighborhoods for the topology of $C^\infty([0,\T],\R^p)$ recalled in Remark~\ref{R:defVueps}. Let 
$u\in C^\infty([0,\T],\R^p)$, $k\in \N$, $\varepsilon_0,\dots,\varepsilon_k$ positive constants.
Let us find an element of $Z'(\T,\Kp)$ inside $V(u,\varepsilon_0,\dots,\varepsilon_k)$.

We complete the sequence $(\varepsilon_j)_{j\in \N}$ into an arbitrary sequence (such as $\varepsilon_j=1$ for all $j> k$ for instance). 
For $j\geq 0$, we set $B_j$ to be the open ball of center $0$ and radius $|u^{(j)}(0)|+\varepsilon_j$.
Since
Assumptions \ref{ass:obs0}, \ref{ass:immersion0}
hold,
we apply Proposition~\ref{P:jets} to find $k'>0$ and $W$ an open and dense subset of $B_0\times B_1\times\cdots\times B_{k'-1}$, such that if $\sigma\in W$, 
for any $\mu \in [0,1]$, $\RR\in \O(p)$, $(x_a,x_b,\mu \RR\cdot\sigma)\notin  A'_k$.
We can assume that $k'\geq k+1$, up to replacing $W$ with $W\times B_{k'+1}\times \cdots\times B_k$.

Let $\sigma^u=(\sigma_0^u,\dots,\sigma_{k'-1}^u)\in (\R^p)^{k'}$ be the $(k'-1)$-jet of $u$ at $0$. With $\sigma=(\sigma_0,\dots,\sigma_{k'-1})$, we define the input
$$
v^\sigma(t)=u(t)+\sum_{j=0}^{k'-1} (\sigma_j-\sigma_j^u)t^j/j!
$$
It is clear that $\omega:\sigma\mapsto  v^\sigma$ is continuous and allows to define $Q={\omega}^{-1}(V(u,\varepsilon_0,\dots,\varepsilon_k))$, an open neighborhood of $\sigma^u$.
Furthermore, $Q\subset B_0\times B_1\times\dots\times B_{k}\times (\R^p)^{k'-k-1}$. Since $W$ is an open and dense subset of $B_0\times B_1\times\dots\times B_{k'-1}$, there exists $\sigma\in W\cap Q$, which by construction satisfies $v^\sigma \in V(u,\varepsilon_0,\dots,\varepsilon_k)$.
By Lemma~\ref{L:good_jet_to_good_input}, we then also have that $v^\sigma\in Z'(\T,\Kp)$, hence the density of $Z'(\T,\Kp)$.
\end{proof}

Let us consider the sequence 
\begin{equation}
    \label{E:KN}
    \Kp_N=\left\{(x_a,x_b)\in \KK^2\mid |x_a-x_b|\ge \frac{\max_{\KK^2} |x_a-x_b|}{N+3}\right\},\qquad N\in \N.
\end{equation}
We can assume that $\KK$ is not reduced to a single point, otherwise $\UU_\KK$ is trivially equal to $C^\infty([0,\T],\R^p)$. This implies that $\Kp_N$ is never empty. From Lemma~\ref{L:ZisDense}, $Z'(\T,\Kp_N)$ is dense in $C^\infty([0,\T],\R^p)$ for any $N\in \N$. By Lemma~\ref{L:ZTKopen}, $Z'(\T,\Kp_N)$ is also open.

\begin{lemma}\label{L:intersect}
If $\KK$ is compact  then $\bigcap_{N\in \N}Z'(\T,\Kp_N)\subset \UU_\KK$.
\end{lemma}

\begin{proof}
The intersection $\bigcap_{N\in \N}Z'(\T,\Kp_N)$ is nonempty as an intersection of open and dense sets for the compact-open topology (it satisfies the Baire category theorem). Let $v\in \bigcap_{N\in \N}Z'(\T,\Kp_N)$, let us check that $v\in \UU_\KK$.

Let $\bar x_a,\bar x_b$ be such that $(\bar x_a,\bar x_b)$ realizes $\max_{\KK^2} |x_a-x_b|$. If $x\in \KK$ is such that $|x-\bar x_a|<|\bar x_a-\bar x_b|/3$, we automatically get by triangular inequality that  $|x-\bar x_b|\geq 2 |\bar x_a-\bar x_b|/3$. Hence the projected set $\uKp_0$ (see \eqref{E:defuKp}) is  equal to $\KK$.
As a consequence, having $v\in Z'(\T,\Kp_0)$ implies that 
there exist 
$T_0\in (0,\T]$ and $q_0\in \N$ for which $x\mapsto \mathcal{H}_{q_0}(x,t,\mu\RR v)$ is an immersion 
over $\KK$ for
all $t\in[0,T_0]$, all $\mu\in[0,1]$ and all $\RR\in \O(p)$.
Applying Lemma~\ref{lem:iminj}, item~(\ref{item2}), thanks of the compactness of $[0,T_0]\times [0,1]\times \O(p)$, there exists $\rho_0,\eta_0>0$ such that for all $t\in[0,T_0]$, all $\mu\in[0,1]$, all $\RR\in \O(p)$, and all $(x_a,x_b)\in \Kp$ such that $|x_a-x_b|<\eta_0$ we have
\[
\left|
    \HH_{q_0}(x_a,t,\mu\RR v)
    -\HH_{q_0}(x_b,t,\mu\RR v)
\right|
\geq 
\rho_0|x_a-x_b|.
\]
Let $N^*\in \N$ be such that $(N^*+3)\eta_0 >\max_{\KK^2} |x_a-x_b|$. From $v\in Z'(\T,\Kp_{N^{*}})$, there exists   $T^{*}\in (0,\T]$, $q^{*}\in \N$ for which $(x_a,x_b)\mapsto \mathcal{H}_{q^{*}}(x_a,t,\mu\RR v)-\mathcal{H}_{q^{*}}(x_b,t,\mu\RR v)$ never vanishes over $\Kp_{N^{*}}$ (for
any $t\in[0,T^*]$, any $\mu\in[0,1]$ and any $\RR\in \O(p)$). Applying Lemma~\ref{lem:iminj}, item~(\ref{item1}), again from the compactness of $[0,T^*]\times [0,1]\times \O(p)$, there exists $\rho^*>0$ such that for all 
$t\in[0,T^*]$, all $\mu\in[0,1]$, all $\RR\in \O(p)$, and all $(x_a,x_b)\in \Kp$ such that $|x_a-x_b|\geq \eta_0$ we have
\[
\left|
    \HH_{q^*}(x_a,t,\mu\RR v)
    -\HH_{q^*}(x_b,t,\mu\RR v)
\right|
\geq 
\rho^*|x_a-x_b|.
\]
Take $T=\min(T_0,T^{*})>0$, $q=\max(q_0,q^{*})$, $\rho=\min(\rho_0,\rho^{*})>0$ and we get for all
$(x_a,x_b)\in \KK^2, t\in [0,T], \mu\in  [0,1],\RR \in  \O(p)$:
\[
\left|
    \HH_{q'}(x_a,t,\mu\RR v)
    -\HH_{q'}(x_b,t,\mu\RR v)
\right|
\geq 
\rho|x_a-x_b|,
\quad 
\forall (x_a,x_b)\in \KK^2, t\in [0,T'], \mu\in  [0,1],\RR \in  \O(p).
\]
Since we also have the immersion property on $K$ (again thanks to $\uKp_0=\KK$), we have proved that $v\in  \UU_\KK$.
\end{proof}

Lemmas~\ref{L:ZisDense} and \ref{L:intersect} immediately imply Theorem~\ref{lem:sus_bis}-{\it (i)}.
The argument we use to prove Theorem~\ref{lem:sus_bis}-{\it (ii)} is similar to the proof of \cite[Theorem 2.1]{MR0541865}. The cited theorem only proves existence of analytic universal inputs but a few elements from Theorem~\ref{lem:sus_bis}-{\it (i)} extend Sussmann's argument to also provides density.

\begin{proof}[Proof of Theorem~\ref{lem:sus_bis}]
Assume $\KK$ is not reduced to a point, under which condition $\UU_\KK=C^\infty([0,\T],\R^p)$.
As a consequence of Lemmas~\ref{L:ZisDense} and \ref{L:intersect}, $\bigcap_{N\in \N}Z'(\T,\Kp_N)$  is a  nonempty intersection of open and dense subsets of $C^\infty([0, \T], \R^p)$, and point {\it (i)} stands.
What remains to show is that {\it (ii)} holds.

Let $u\in C^\infty([0,\T],\R^p)$, $k\in \N$, $\varepsilon_0,\dots,\varepsilon_k$ positive constants, and the base neighborhood
$V(u,\varepsilon_0,\dots,\varepsilon_k)$.
Let us prove that there exists an input $v\in \UU_\KK\cap C^\omega([0,\T],\R^p)\cap V(u,\varepsilon_0,\dots,\varepsilon_k)$.

We first apply Weierstrass density theorem to find a polynomial $P$ belonging to $V(u,\varepsilon_0/2,\dots,\varepsilon_k/2)$. Then we define $(B'_i)_{i\in \N}$ a sequence of open balls centered at $P^{(i)}(0)$ and of radii 
$r=\min(\varepsilon_0,\dots,\varepsilon_k) \e^{-T}/4$. That way for any input $w\in C^{\omega}([0,T],\R^p)$ such that $w^{(i)}(0)\in B'_i$, we have 
$$
P(t)-w(t)=\sum_{i=0}^{\infty }\frac{P^{(i)}(0)-w^{(i)}(0)}{i!}t^i,
$$
which implies $w\in V(P,\varepsilon_0/4,\dots,\varepsilon_k/4)\subset V(u,\varepsilon_0,\dots,\varepsilon_k)$.

Let $(B_i)_{i\in \N}$ be a sequence of open balls centered at 0 and of radii 
$P^{(i)}(0)+2r$. 
For every $N\in \N$, we apply Proposition~\ref{P:jets} on the compact set $\Kp_N$ to find $q_N$, an open and dense $W_N\subset B_0\times \cdots\times B_{q_N-1}$ such that for all $(x_a,x_b)\in \Kp_N$, $\sigma\in W_N$, $\mu\in [0,1]$, $\RR\in \O(p)$,	$(x_a,x_b,\mu \RR\cdot\sigma)\notin  A'_q$. Naturally, this implies that if the $(q-1)$-jet of the input $w$ at $t=0$ belongs to $W_N$ then $w\in Z'(\T,\Kp_N)$ from Lemma~\ref{L:good_jet_to_good_input}.
We complete each $W_N$ into $W'_N=W_N\times \prod_{i=q}^{\infty} B_i$. Each $W'_N$ is an open and dense subset of $\prod_{i=0}^{\infty} B_i$ endowed with the product topology. Hence $\bigcap_{N=0}^{\infty} W'_N$ is a countable intersection of open and dense subsets. The product topology satisfies the Baire category theorem and $\bigcap_{N=0}^{\infty} W'_N$ is dense in $\prod_{i=0}^{\infty} B_i$.
Since the product $\prod_{i=0}^{\infty} B'_i $ is open in the product space $\prod_{i=0}^{\infty} B_i$, the density of $\bigcap_{N=0}^{\infty} W'_N$ implies that the intersection $\bigcap_{N=0}^{\infty} W'_N\cap \prod_{i=0}^{\infty} B'_i $ is nonempty. For any $\sigma$ in that intersection, set
$v(t)=\sum_{i=0}^\infty\frac{\sigma_i}{i!}t^i$. We have by construction 
\[
v\in V(u,\varepsilon_0,\dots,\varepsilon_k)\cap\left(
\bigcap_{N=0}^\infty Z'(\T,\Kp_N)
\right)
\subset V(u,\varepsilon_0,\dots,\varepsilon_k)\cap \UU_\KK
\]
which proves the density and concludes the argument.
\end{proof}

\section{Conclusion and perspective}
In this article, we have introduced a new method for designing stabilizing output feedback law for nonlinear (analytic) control systems. 
This approach is based on the use of a control template, which consists 
\startmodif
in an input making the system observable over a finite time interval. The control input applied to the system is perodically ``templated'' on it, up to a rescaling and an isometry.
\stopmodif
\startmodif
Thanks to this templating, the state of the input-output system retains observability, enabling the use of the high-gain observer.
\stopmodif
The initial result we achieve is to demonstrate that the existence of these control templates and (strong differential) observability at the target, along with the existence of a globally asymptotically and locally exponentially stabilizing state feedback loop at the origin, allow the derivation of a hybrid control law that semi-globally stabilizes the origin through dynamic output feedback.
Subsequently, we establish by extending Sussman's universality result \cite{MR0541865} that, within our context,
\startmodif
control templates are generic among smooth inputs.
\stopmodif
The combination of these two results  allows us to obtain the first generic semiglobal dynamic output feedback stabilization result for nonlinear control systems.
We do not rely on any uniform observability assumption \cite{TeelPraly1994, jouan1996finite}, nor on any passivity or dissipativity properties \cite{PRIEUR2004847, SACCHELLI20204923}.

It remains an open question whether a similar result could be achieved by using a two modes strategy assuming exponential stabilizability and observability of the null input as in \cite{shim2003asymptotic} (the potential difficulty lies in passing from practical to exact stablization), or as in \cite{coron1994stabilization} (it lies in passing from local to semiglobal stablization). In any case, the template control method has the advantage of not relying on a two modes separation: estimation and observation are achieved simultaneously, which makes it closer to usual observer based methods.
The observability assumption we make is still stronger than the ones imposed by Coron in \cite{coron1994stabilization} that are almost necessary (see \cite{brivadis2021new}).
We believe that methods based on template control could be further developed to address other observability issues, including in particular the problem of stabilization at an unobservable target.

\appendix

\section{Additional technical lemmas}\label{sec:tec}

\begin{lemma}\label{lem:iminj}
Let $n$ and $m$ be two positive integers,
$\Kp$ be a compact subset of $\R^n\times \R^n$,
$\Pi$ be a  topological space,
and
$F:\R^n\times\Pi\to\R^m$ be a continuous map. Define $F_\pi:= x\mapsto F(x, \pi)$,
for all $\pi\in\Pi$.

Recall that $\diag_\eta$ and $\uKp$ are defined by \eqref{E:defDeta} and \eqref{E:defuKp}, respectively.

\begin{enumerate}[(i)]
    \item\label{item1} There exists a continuous function $\rho_1:\Pi\times \R_+\to \R_+$ such that for all $(x_a, x_b, \pi)\in(\Kp\setminus\mathring{\diag}_\eta)\times\Pi$,
    \begin{equation}
        |F_\pi(x_a)-F_\pi(x_b)|\geq \rho_1(\pi,\eta) |x_a-x_b|,
    \end{equation}
    and if $(x_a,x_b)\mapsto F_\pi(x_a)-F_\pi(x_b) $ never vanishes over $\Kp\setminus \mathring{\diag}_\eta$ for any $\pi\in\Pi$, then $\rho_1$ is positive.

    \item \label{item2} Assume $F_\pi$ is $C^1$ over $\uKp$ and that $(x, \pi)\mapsto \frac{\partial F_\pi}{\partial x}(x)$ is continuous. There exists two continuous functions $\eta_2:\Pi\to \R_+$ and $\rho_2:\Pi \to \R_+$ such that for all $\pi\in \Pi$, for all $(x_a, x_b)\in\Kp\cap\diag_{\eta_2(\pi)}$, 
    \begin{equation}
        |F_\pi(x_a)-F_\pi(x_b)|\geq \rho_2(\pi) |x_a-x_b|,
    \end{equation}
    and if $\frac{\partial F_\pi}{\partial x}(x)$ is injective for all $(x, \pi)\in\uKp\times\Pi$, then $\eta_2$ and $\rho_2$ are positive.

    \item \label{item3} If $\Pi$ is compact and $F_\pi$ is an injective immersion over a compact $\KK\subset \R^n$ for all $\pi\in\Pi$, then there exists $\rho>0$ such that
    \begin{equation}\label{eq:Finj}
        |F_\pi(x_a)-F_\pi(x_b)|\geq \rho |x_a-x_b|,
    \end{equation}
    for all $(x_a, x_b, \pi)\in\KK^2\times\Pi$, and there exists a continuous map $F^\invert: \R^m\times\Pi\to\R^n$ 
    that is globally Lipschitz with respect to its first variable, uniformly with respect to the second, and
    such that
    $F^\invert(F(x, \pi),\pi) = x$ for all $(x,\pi)\in\KK\times\Pi$.
\end{enumerate}

\end{lemma}
\begin{proof}
This lemma is an adaptation of \cite[Lemma 3.2]{doi:10.1137/110853091}, where we investigate the dependency of $\rho_1$ and $\rho_2$ with respect to the parameter $\pi$. The last item is a direct consequence of \cite[Lemma A.12]{bernard2019observer}, but we give a proof for the sake of completeness.

\begin{enumerate}[{\it(i)}]

\item The result is obvious if $\Kp$ is empty (pick $\rho_1 = 1$).
Otherwise, for all $\pi\in\Pi$ and all $\eta\in [0, d]$ with $d:=\max_\Kp |x_a-x_b|$, define
$$\rho_1(\pi, \eta) = \min_{(x_a, x_b)\in\Kp\setminus \mathring{\diag}_\eta}\frac{|F_\pi(x_a)-F_\pi(x_b)|}{|x_a-x_b|}.$$
Note that $\rho_1$ is well-defined and continuous since $F$ is continuous and $\Kp\setminus \mathring{\diag}_\eta$ is a nonempty compact set since $\eta\leq d$.
Moreover, $\rho_1$ is positive if $(x_a,x_b, \pi)\mapsto F_\pi(x_a)-F_\pi(x_b) $ does not vanish on $(\Kp\setminus \mathring{\diag}_\eta)\times\Pi$.
If $\eta>d$, define $\rho_1(\pi, \eta) = \rho_1(\pi, d)$.

    \item The result is obvious if $\Kp$ is empty (pick $\eta_2=1$ and $\rho_2 = 1$), hence we now assume $\Kp$ nonempty.
Since $(x, \pi)\mapsto \frac{\partial F_\pi}{\partial x}(x)$ is continuous, there exists a continuous function $\mathcal{E}:(\R^n)^2\times\Pi\to\R^n$ such that
$\mathcal{E}(0, 0, \pi) = 0$ for all $\pi\in\Pi$, and we have
for all $(x_a, x_b)\in (\R^n)^2$,
$$
F_\pi(x_a)-F_\pi(x_b) = \frac{\partial F_\pi}{\partial x}(x_b)(x_a-x_b) + |x_a-x_b| \mathcal{E}(x_a, x_b, \pi).
$$
Define
$$
\rho_2(\pi) = \frac{1}{2}\min_{\substack{(x,\xi)\in\uKp\times\R^n\\|\xi|=1}} \left|\frac{\partial F_\pi}{\partial x}(x)\xi\right|.
$$
Note that $\rho_2$ is well-defined and continuous since $(x,\pi)\mapsto\frac{\partial F_\pi}{\partial x}(x)$ is continuous and $\uKp$ is a nonempty compact set.
If $\frac{\partial F_\pi}{\partial x}(x)$ is injective for all $(x, \pi)\in\uKp\times\Pi$, then $\rho_2$ is positive.

Define 
$$
\eta_2(\pi) = \min_{\substack{(x_a,x_b)\in\Kp\\|\mathcal{E}(x_a, x_b, \pi)|\geq \rho_2(\pi)}}|x_a-x_b|.
$$
Note that $\eta_2$ is well-defined and continuous since $\mathcal{E}$ is continuous, $\mathcal{E}(0, 0, \pi) = 0$, and $\Kp$ is a nonempty compact set.
Moreover, $\eta_2(\pi)>0$ if $\rho_2(\pi)>0$.
If $(x_a, x_b)\in \Kp\cap\diag_{\eta_2(\pi)}$, then $|\mathcal{E}(x_a, x_b, \pi)|\leq\rho_2(v)$, which yields
$$
|F_\pi(x_a) - F_\pi(x_b)| \geq 2\rho_2(v)|x_a-x_b| - |x_a-x_b|\mathcal{E}(x_a, x_b, \pi) \geq \rho_2(v)|x_a-x_b|.
$$

\item 
For each $\pi\in\Pi$,
the map $F_\pi$ is injective, hence admits a left inverse
$F^{\invert}(\cdot, \pi): F_\pi(\KK)\to\KK$.
Applying \eqref{item1} and \eqref{item2} with $\Kp = \KK^2$, define $\eta = \min_\Pi \eta_2$ and $\rho = \min_\Pi (\rho_1(\cdot, \eta), \rho_2)$. Then \eqref{eq:Finj} holds for all $(x_a, x_b, \pi)\in\KK^2\times\Pi$, hence
$F^{\inv}(\cdot, \pi)$ is Lipschitz continuous over its domain of definition, with Lipschitz constant $\kappa:=\frac{1}{\rho}$.
Hence, according to \cite[Theorem 1]{mcshane}, $F^{\invert}(\cdot, \pi)$ admits an extension $F^{\invert}(\cdot, \pi):\R^{m}\to\R^n$, defined by
$$
F^{\invert}(z, \pi) =
(
\min_{\tilde z\in F_\pi(\KK)} \{F^{\invert}_j(z, \pi)+\kappa |\tilde z-z |\}
)_{1\leq j\leq n}
$$
for all $(z,\pi)\in\R^m\times\Pi$, where $F^{\invert}_j$ is the $j$th-coordinate of $F^{\invert}$, such that $F^\invert$ is continuous
and globally Lipschitz continuous
with respect to $z$
with constant $\kappa\sqrt{n}$, hence uniformly with respect to $\pi$.

\end{enumerate}

\end{proof}

\begin{lemma}\label{lem:R0}
For all $u_0, v_0\in\R^p$ and all $\RR_{u_0}\in \RRR_0(u_0)$, there exists $\RR_{v_0}\in \RRR_0(v_0)$ such that
\begin{equation}\label{eq:R0}
\||u_0|\RR_{u_0}-|v_0|\RR_{v_0}\|\leq |u_0-v_0|
\end{equation}
\end{lemma}
\begin{proof}
When $p=1$, the result follows by picking $\RR_{v_0} \in \{-1,1\}$ such that $|v_0|\RR_{v_0}= v_0$. Now, assume $p>1$.
If $v_0 = 0$, simply set $\RR_{v_0} = \RR_{u_0}$.
If $u_0=0$,
\eqref{eq:R0} holds for any $\RR_{v_0}\in\RRR_0(v_0)$.
Now, assume that $u_0\neq0$ and $v_0\neq0$.
Let $\tilde\RR\in\O(p)$ be such that
$\tilde\RR u_0$ and $\tilde\RR v_0$ belong to  $\mathrm{span}\{e_1, e_2\}$
with $e_1=(1,0,0,\dots,0)$ and $e_2=(0,1,0,\dots,0)$.
Let
$\bar\RR=\mathrm{diag}(J, 0,\dots,0)
$
with $J=
\begin{pmatrix}
\cos(\psi)&-\sin(\psi)\\
    \sin(\psi)&\cos(\psi)
\end{pmatrix}$
and $\psi\in(-\pi, \pi]$
be such that
$\bar\RR\tilde\RR u_0 = \tilde\RR v_0$
.
Define $\RR_{v_0} =
\tilde\RR^{-1}
\bar\RR
\tilde\RR
\RR_{u_0}$.
Then
$\RR_{v_0}(|v_0|,0,\dots,0) = v_0$
and
\begin{align*}
\big\||u_0|\RR_{u_0}-|v_0|\RR_{v_0}\big\|
&=
\big\||u_0|\RR_{u_0}-|v_0|
\tilde\RR^{-1}
\bar\RR
\tilde\RR
\RR_{u_0}
\big\|
\\
&\leq
\big\|
|u_0|\tilde\RR^{-1}\tilde\RR-|v_0|
\tilde\RR^{-1}
\bar\RR
\tilde\RR
\big\|
\\
&\leq
\big\||u_0|\Id_{\R^p}-|v_0|
\bar\RR
\big\|
\\
&\leq
\max\left(
\big\||u_0|\Id_{\R^2}-|v_0|
J
\big\|
, ||u_0|-|v_0||\right)
\\
&\leq
\max\left(
\sqrt{|u_0|^2+|v_0|^2-2|u_0||v_0|\cos(\psi)},
|u_0-v_0|\right)
\\
&\leq
|u_0-v_0|.
\end{align*}
\end{proof}

\section{Proofs of Lemmas \ref{lem:analytic}, \ref{lem:hyp12} and \ref{Ass_HG}}
\label{sec:tec2}

\begin{proof}[Proof of Lemma \ref{lem:analytic}]
For all $i\in\N$, define
$E_i := \{(x_a, x_b) \in \R^n\times\R^n\mid \HH_i(x_a, 0, u) = \HH_i(x_b, 0, u)\}$.
Then $(E_i)_{k\in\N}$ is a non-increasing family of analytic sets.
According to \cite[Chapter 5, Corollary 1]{Narasimhan}, $(E_i \cap \KK_x^2)_{i\in\N}$ is stationary, \emph{i.e.}, there exists $k\in\N$ such that $E_k \cap \KK_x^2 = E_i \cap \KK_x^2$ for all $i\geq k$. Hence,
$E_k \cap \KK_x^2=\bigcap_{i\in\N} E_i\cap \KK_x^2$.
Under the analyticity assumptions, for any pair $x_a,x_b\in \KK_x$, having $h\circ X(x_a,\cdot ,u)\neq h\circ X(x_b,\cdot ,u)$ implies that time derivatives at $t=0$ of these two functions must differ at some rank, i.e., there exists $i$ such that $\HH_i(x_a, 0, u)\neq \HH_i(x_b,0,u)$.
Since $u$ is analytic, the observability assumption yields $\bigcap_{i\in\N} E_i \cap \KK_x^2 = \{(x, x), x\in\KK_x\}$.
Thus $E_k \cap \KK_x^2=\{(x, x), x\in\KK_x\}$, i.e. \eqref{syst} is differentially observable of order $k$ over $\mathring \KK_x$.
\end{proof}

\begin{proof}[Proof of Lemma \ref{lem:hyp12}]
By Lemma~\ref{lem:analytic}, Assumption~\ref{ass:obs0} implies that system \eqref{syst} is differentially observable over $\KK_x$ of some order $k_1\in\N$.
Under Assumption \ref{ass:immersion0}, for any $x\in\R^n$, let $U_x\subset\R^n$ be an open set such that $\frac{\partial \HH_{k_x}(\cdot, 0, 0)}{\partial x}$ is injective over $U_x$. Such an open set exists since $x\mapsto \frac{\partial \HH_{k_x}(\cdot, 0, 0)}{\partial x}(x)$ is continuous over $\XX$ and the set of injective matrices in $\R^{m(k+1)\times n}$ is open.
Since $(U_x)_{x\in\KK_x}$ is an open cover of the compact set $\KK_x$, one can extract a finite subcover $(U_{x_i})_{i\in\II}$, with $\II$ a finite set and $x_i\in\KK_x$ for each $i\in\II$.
Set $k_2 = \max_{i\in\II} k_{x_i}$. 
Note that for each $x\in U_{x_i}$ and each $k\in\N$,
$\frac{\partial \HH_{k}(\cdot, 0, 0)}{\partial x}$ is injective for all $k\geq k_{x_i}$.
Hence, for all $x\in\KK_x$,
$\frac{\partial \HH_{k_2}(\cdot, 0, 0)}{\partial x}$ is injective.
Finally, set $k = \max(k_1, k_2)$.
By definition of $k_1$ and $k_2$, system \eqref{syst} is strongly differentially observable of order $k$.
\end{proof}

\begin{proof}[Proof of Lemma \ref{Ass_HG}]
It is a direct corollary of Lemma \ref{lem:iminj}-\eqref{item3}
applied to the map $\HH_q$, since $x\mapsto \HH_q(x, t, \mu\RR v^*)$ 
is an injective immersion over the compact set $\KK_x'$ for all $(t, \mu, \RR)$ in the compact space $[0,T]\times[0, \bar \lambda]\times \O(p)$.
\end{proof}

\bibliographystyle{abbrv}
\bibliography{biblio}
\end{document}